\documentclass[12pt]{amsart}
\usepackage{amsmath}

\newtheorem{thm}{Theorem}[section]

\newtheorem{lem}[thm]{Lemma}%[section]
\newtheorem{prop}[thm]{Proposition}%[section]
\newtheorem{maintheo}{{\sc Theorem}}
 \newtheorem{maincor}[maintheo]{{\sc Corollary}}
\newtheorem{maindefn}[maintheo]{{\sc Definition}}

\theoremstyle{definition}
\newtheorem{defn}{Definition}[section]
\newtheorem{conj}{Conjecture}
\newtheorem{rem}{Remark}[section]

\theoremstyle{remark}
\theoremstyle{plain}

\newcommand{\Z}{{\mathbf Z}}

\newcommand{\ccal}{{\mathcal C}}

\newcommand{\fcal}{{\mathcal F}}
\newcommand{\hcal}{{\mathcal H}}

\newcommand{\dcal}{{\mathcal D}}

\newcommand{\lcal}{{\mathcal L}}

\newcommand{\rcal}{{\mathcal R}}
\newcommand{\kcal}{{\mathcal K}}
\newcommand{\wcal}{{\mathcal W}}
\newcommand{\scal}{{\mathcal S}}

\newcommand{\X}{{\mathbf X_{\Gamma}}}

\newcommand{\D}{\mathbf D}

\newcommand{\R}{{\mathbf R}}
\newcommand{\C}{{\mathbf C}}
\newcommand{\Lap}{\triangle}
\newcommand{\To}{\longrightarrow}

\newcommand{\Tr}{\operatorname{Tr}}
\newcommand{\Vol}{\operatorname{Vol}}
\newcommand{\hs}{\operatorname{HS}}

\newcommand{\half}{\frac{1}{2}}

\newcommand{\Op}{\operatorname{Op}}
\newsymbol\upharpoonright 1316
\newcommand{\la}{\langle}
\newcommand{\ra}{\rangle}

\newcommand{\defi}{\stackrel{\rm def}{=}}
\providecommand{\norm}[1]{\lVert#1\rVert}

\def\1{{{\mbox{${\mathrm{1\negthinspace\negthinspace I}}$}}}}

\numberwithin{equation}{section}

\begin{document}
%Baselineskip 18pt

\title[Geodesic flow and Schr\"odinger flow on hyperbolic surfaces]
{Intertwining the geodesic flow and the Schr\"odinger group on hyperbolic surfaces}

\author{Nalini Anantharaman and  Steve Zelditch}
\address{Laboratoire de Math\'ematique, Universit\'e d'Orsay Paris XI, 91405 Orsay Cedex, France}
\address{Department of Mathematics, Northwestern University,
Evanston IL,  60208-2730, USA}

\thanks{Research partially supported by NSF grant
  DMS-0904252. NA wishes to acknowledge the support of Agence Nationale de la Recherche,
under the grants ANR-09-JCJC-0099-01 and ANR-07-BLAN-0361.}

\maketitle

%\tableofcontents

\begin{abstract} We construct an explicit intertwining operator
$\lcal$ between the Schr\"odinger group $e^{ it \frac\Lap2} $ and
the geodesic flow $g^t$ on certain Hilbert spaces of symbols on
the cotangent bundle $T^* \X$ of a compact hyperbolic surface $\X
= \Gamma \backslash \D$. We also define $\Gamma$-invariant
$g^t$-eigendistributions $PS_{j, k, \nu_j,-\nu_k}$
(Patterson-Sullivan distributions) out of pairs of
$\Lap$-eigenfunctions, generalizing the diagonal case $j = k$ in
\cite{AZ}. The operator $\lcal$ maps  $PS_{j, k, \nu_j,-\nu_k}$ to
the Wigner distribution $W^{\Gamma}_{j, k}$ studied in quantum
chaos. We define Hilbert spaces $\hcal_{PS}^*$  spanned by
$\{PS_{j, k, \nu_j,-\nu_k}\}$, resp. $\hcal_W^*$ spanned by
$\{W^{\Gamma}_{j, k}\}$, and show that $\lcal$ is a unitary
isomorphism from $\hcal_{PS}^* \to \hcal_{W}^*. $ The symbols lie
in the dual Hilbert spaces.

\end{abstract}

\section{Introduction}

On a hyperbolic surface $\X = \Gamma \backslash \D$, there is an
intimate relation between the spectral properties of the laplacian
$\Lap$  and those of the geodesic flow $g^t$ on the unit tangent
bundle $S \X$. The Selberg trace formula  gives an exact formula
for the trace of the Schr\"odinger  flow $e^{ it \frac\Lap2} $ as
a sum over closed geodesics, and it may be interpreted as a trace
of the pull-back operator by $g^t$ \cite{G}. Equivalently,
eigenvalues of $\Lap$ are (re-parameterizations of) the resonances
of $g^t$ (see \cite{B} for background); see also \cite{Bis,Po2,M}
for some of the many different perspectives on this relation. In
this article, we give a yet stronger relation between the two
flows~: we construct an explicit intertwining operator $\lcal$
(Definition \ref{LDEF}) between the Schr\"odinger flow and the
geodesic flow. Our main result, Theorem \ref{INTintrodual}, is
that there exist Hilbert spaces of symbols on which $\lcal$ is a
unitary intertwining operator between the classical and quantum
flow. Much of the problem is to construct the appropriate Hilbert
spaces, which we denote by $\hcal_W, \hcal_{PS}$ (see Definitions
\ref{HCALWDEF}-\ref{HCALPSDEF}). They cannot be the standard
Hilbert spaces, $L^2(S \X)$ for $g^t$, resp. Hilbert-Schmidt
operators for $e^{ it \frac\Lap2} $, since the spectrum of $g^t$
is continuous, while that of $e^{ it \frac\Lap2} $ is discrete;
and they are also quite different from the Banach spaces
constructed in \cite{BT,BKL,BL,FRS,GL} in the theory of resonances
of $g^t$.

The construction of $\lcal, \hcal_W, \hcal_{PS}$ and the proof of
the intertwining property grow out of our previous work \cite{AZ},
where we introduced and studied a family of distributions (that we
called Patterson-Sullivan distributions) on the unit tangent
bundle of a hyperbolic surface. These distributions are invariant
under the geodesic flow, and we showed that they are closely
related to the {\em Wigner distributions} appearing in the theory
of quantum ergodicity. The Patterson-Sullivan distributions are
naturally constructed from the family of eigenfunctions of the
laplacian, and we showed that they also arise as residues of
dynamical zeta-functions at the poles located on the critical
line.

In this paper, we introduce the family of {\em off-diagonal} Patterson-Sullivan distributions,
and show how they are related to the off-diagonal Wigner
distributions (appearing in the study of quantum mixing).
This construction is a rather straightforward generalization
of the work done in  \cite{AZ}. More importantly, we show that
 these formulae directly lead to an operator intertwining the geodesic
 flow and the Schr\"odinger flow on the hyperbolic plane (or a compact
  quotient $\X$). Roughly speaking, the dual Hilbert spaces
  $\hcal^*_W, $ resp. $\hcal^*_{PS}$, intertwined by $\lcal$  are spanned by the Wigner,
  resp. Patterson-Sullivan, distributions.
  The main goal of this paper is to construct explicitly this intertwining operator,
   first on the hyperbolic plane, then on a compact quotient, and
to investigate some of its properties. The existence of this operator opens
the way to the construction of a quantization procedure satisfying the Egorov theorem in its exact form (without remainder term).

We have to explain in what sense one can find an intertwining operator between
the geodesic flow and the Schr\"odinger flow. The former acts on functions on the
 (co)tangent bundle $T\X$ whereas the latter acts on functions on the base manifold $\X$. In fact,
 we let the Schr\"odinger group act on the space of operators, by conjugation (as in the Heisenberg picture
 of quantum mechanics). Operators have a Schwartz kernel, which is a distribution on the product $\X\times \X$.
 By taking the local Fourier transform of the kernel with respect to the second component, we get a distribution
  on the cotangent bundle $T^*\X$, called the {\em symbol} of the operator. This way, we see that the Schr\"odinger
  group acts naturally on the space of distributions on $T^* \X$ (in the paper we will always identify the tangent
  and the cotangent bundles by means of the riemannian metric). With this formulation, the Schr\"odinger flow acts
   on the same space as the geodesic flow, and it is in this sense that we shall intertwine their actions.

\subsection{Notation} To state our results, we need to introduce some notation (see \S \ref{BACKGROUND} for
more details).  We will denote $G = PSU(1,1) \simeq PSL(2, \R)$,
$K=PSO(2, \R)$ a maximal compact subgroup, and $ G/K$ the
corresponding symmetric space, for which we will in general use
the picture of the hyperbolic disc $\D=\{z\in\C, |z|<1\}$, endowed
with the riemannian metric
$$ds^2=\frac{4|dz|^2}{(1-|z|^2)^2}.$$
This is a standard normalization in hyperbolic geometry, but we
caution that it differs by a constant factor from the
normalization used by Helgason \cite{He}; for us, the
$L^2$-spectrum of the laplacian on $\D$ is $(-\infty, -\frac14]$.
Hence some discrepancies between some of our formulae and
Helgason's.

 It is well-known that $G$ can be identified with the unit tangent bundle $S\D$ of the hyperbolic disc $\D$ (when using the theory of pseudodifferential operators, it is more natural to work on the cotangent bundle, but we will always identify both). We will be particularly interested in the geodesic flow, which acts on $G$ by right multiplication as follows~: for all $g\in  PSL(2, \R)$, for all $t\in\R$,
$g^t(g)=ga_t$ where
$a_t = \left( \begin{array}{ll} e^{t/2} & 0 \\ & \\
0 & e^{-t/2} \end{array} \right)\in SL(2, \R).$  We will also use the action of the horocycle
flow $(h^u)_{u\in\R}$, acting by $h^u(g)=gn_u$ where
$n_u = \left( \begin{array}{ll} 1 & u \\ & \\
0 & 1\end{array} \right)\in SL(2, \R). $

We will work with two special parameterizations of the unit
tangent bundle, identified with $G$. The first one is obtained by
writing $G\sim (G/K)\times K\sim \D\times B$ where $B$ is the
boundary at infinity of $\D$, identified with the unit circle
$S^1$ in the Poincar\'e disc model. The group $K$ and the boundary
at infinity $S^1$ are identified by the map
$$
 \left( \begin{array}{ll} \cos \theta & -\sin \theta \\ & \\
\sin\theta & \cos\theta \end{array} \right)\mapsto e^{2i\theta}.$$
This way, a point in $G$ can be parametrized by the coordinates $(z, b)$, where $z\in \D$ and $b\in B$. In geometric terms, if $(z, b)$ is identified with a unit tangent vector in $S\D$, then $b$
represents the (forward) limit point of the geodesic generated by $(z, b)$. The action of $G$ on itself by left-multiplication yields an action of $G$ on $B$ (Section \ref{s:HA}).

We shall also use the following parameterization~: denote
$B^{(2)}=\{(b', b)\in B\times B, b\not=b'\}$ the set of pairs of
distinct points in the boundary. Each oriented geodesic in $\D$ is
completely determined by its (unique) forward limit point $b$ in
$B$ and its (unique) backward limit point $b'\not= b$ in $B$~: we
will denote $\gamma_{b', b}$ the geodesic going from $b'\in B$ to
$b\in B$. Thus, $B^{(2)}$ can be naturally identified with the set
of oriented geodesics of $\D$. The elements of $G$ can be
parametrized by $(b', b, \tau)$ with $(b', b)\in B^{(2)}$ and
$\tau\in \R$. We identify $(b', b, \tau)$ with the point $(z, b)
\in \D \times B$, where  $z$ is on the geodesic $\gamma_{b', b}$,
situated $\tau$ units from  the point $z_{b, b'} \in \gamma_{b,
b'}$ closest to the origin $o\defi eK\in \D $.

Our final goal is to obtain formulae that are valid on a compact
quotient of $\D$; that is, we consider a co-compact discrete
subgroup $\Gamma\subset G$. We assume it has no
torsion\footnote{This assumption is probably not necessary.} and
contains only hyperbolic elements. Then the quotient
$\X=\Gamma\backslash \D$ is a compact hyperbolic surface.

\subsection{Quantization, Wigner distributions and Patterson-Sullivan distributions}

A quantization procedure adapted to the hyperbolic disc was defined in \cite{Z3}, using Helgason's version of the Fourier transform \cite{He}. For $(z, b)\in \D\times B$, define the Busemann function $\la z, b\ra$ as the signed distance to $o$ of the horocycle going through the points $z\in\D, b\in B$. The family of functions $z\mapsto e^{(\frac12+ir)\la z, b\ra}$ ($r>0, b\in B$) forms a basis of generalized eigenfunctions of the laplacian on $L^2(\D)$ \cite{He}. The hyperbolic pseudodifferential operators introduced by \cite{Z3} are defined by
\begin{equation} \label{e:opZ}\Op(a)e^{(\frac12+ir)\la \bullet, b\ra}=a(\bullet, b, r)e^{(\frac12+ir)\la \bullet, b\ra},
\end{equation}
if $a=a(z, b, r)$ is a function on $\D\times B\times \R\simeq S\D\times \R$ which must have ``reasonable'' decay and smoothness properties (section \ref{PSI}). The function $a$ is called the symbol of the operator. We note that by choosing $r<0$ instead of $r>0$ we obtain another basis of generalized eigenfunctions of the laplacian. We also note that the Schwartz-kernel of the operator is formally given by
\begin{equation} K_a(z,w) =  \int_B \int_{r\in \R_+} a(z, b', r) e^{( \frac12+ir )\langle z,
b' \rangle}  e^{( \frac12- ir )\langle w, b' \rangle} dp(r) db',
\end{equation}
where $dp(r)$ is the Plancherel measure defined in \S
\ref{MEASURES}. Paley-Wiener type theorems relating the decay and
regularity of $a$ and those of $K_a$ will be recalled in \S
\ref{s:Schw}. {\em In most formulae we must assume that $K_a$
decays sufficiently fast away from the diagonal $\{z=w\}$. This
implies in particular that the corresponding symbol $a$ has a
holomorphic extension to $r\in\C$}.

Let $\X$ be a compact quotient of $\D$ as above, and fix an orthonormal basis $(\phi_k)$ of $L^2(\X)$ formed of eigenfunctions of the laplacian. We use the standard notations in hyperbolic spectral theory~: the eigenfunctions $\phi_k$ satisfy
$$\Lap \phi_k=-\left(\frac14-\nu_k^2 \right)\phi_k=-\left(\frac14+r_k^2 \right)\phi_k,$$
where $\nu_k=ir_k\in \R\cup i\R$ is called the spectral parameter (on a compact surface, only a finite number of $r_k$s are imaginary). For each eigenvalue there are two possible choices for the spectral parameter.

If $a$ is $\Gamma$-invariant, it was shown in \cite{Z3, Z} that $\Op(a)$ preserves the space of $\Gamma$-invariant functions. We denote by $\Op_\Gamma(a)$ the operator $\Op(a)$
acting on $\Gamma$-invariant functions.
The ``Wigner distributions\footnote{This terminology is normally used in a euclidean context, but it is sometimes convenient to extend it to non-euclidean geometries.}'' $W_{j, k}$
 are defined on $  S\D \times \R \simeq  G\times \R$ by the formula
\begin{equation}  \int_{S\D\times \R} a \, dW_{j, k}\defi \langle \Op(a )
\phi_{j}, \phi_{k}\rangle_{\D},
\end{equation}
for $a$ a function on  $  S\D \times \R$, with appropriate growth and smoothness properties. The distribution $W_{j, k}$ is invariant by the action of $\Gamma$ on $S\D$, and thus can be used to define a distribution $W_{j, k}^\Gamma$ on the quotient
$  S\X \times \R \simeq \Gamma\backslash G\times \R$~: if $a$ is a function on
$ \Gamma\backslash G\times \R$, in other words a $\Gamma$-invariant function on $G\times \R$,
we define $$ \int_{S\X\times \R} a  \,dW_{j, k}^{\Gamma} =\int_{S\D\times \R} \chi a\,  dW_{j, k}, $$
where $\chi$ is a smooth fundamental cut-off function for the action of $\Gamma$ (see \S \ref{FUND}). It can easily be seen that this definition does not depend on the choice of $\chi$.
Moreover we have
\begin{equation}  \int_{S\X\times \R} a \, dW_{j, k}^{\Gamma} = \langle \Op_\Gamma(a )
\phi_{j}, \phi_{k}\rangle_{L^2(\X)} = \mbox{Tr}_{\X}\; \Op_\Gamma(a) \circ
(\phi_{j} \otimes \phi_{k}^*),
\end{equation}
for $a$ a function on  $  S\X \times \R$, with appropriate growth and smoothness properties (see Section \ref{PSI} for a detailed discussion).

 The Wigner distribution $W_{j, k}^{\Gamma}$ may also be expressed in terms of the boundary values of the eigenfunctions $\phi_j, \phi_{k}$. The boundary values $T_{k, \nu_k}$ of  $\phi_{k}$ is a distribution on the boundary $B$, with the property that
$$\phi_{k}(z) = \int_B e^{(\frac12 + \nu_k) \langle z,
b \rangle } T_{k, \nu_k} (db),$$ for all $z\in\D$. It depends on the choice of a spectral parameter $\nu_k$, and is unique if we pick $\nu_k$ such that
 $\frac{1}{2} + \nu_k \not= 0, -1,
-2,\cdots$ (\cite[Theorems 4.3 and 4.29]{He}; see also \cite{He2}). With a slight abuse of notation we shall denote $T_{ \nu_k}$ instead of $T_{k, \nu_k}$. Using the definition of $\Op$, we have
$$\int_{S\X\times \R} a \, d{W}^\Gamma_{j, k}= \int_{\D\times B} \chi(z, b) a(z, b, r_j)\overline{ \phi_{k}(z)  }e^{(\frac12+ \nu_j) \langle z,
b \rangle } T_{\nu_j} (db)\Vol( dz)$$
where $\chi$ is a smooth fundamental cut-off, see \S \ref{FUND}.

It follows from its definition that $W^{\Gamma}_{j,
k}$ is an eigendistribution of the quantum evolution. Define
\begin{equation} \label{EGORAUT}  \alpha^t (\Op(a)) = e^{- it \frac\Lap2} \Op(a) e^{i t \frac\Lap2}. \end{equation}
We then have
\begin{equation} \label{EIGENREL} \langle \alpha^t (\Op_\Gamma(a)) \phi_{j},
\phi_{k}\rangle =  e^{i t \frac{( \nu^2_j- \overline{\nu_k}^2)}2} \langle \Op_\Gamma(a )
\phi_{j}, \phi_{k}\rangle=e^{i t \frac{(r^2_k - r^2_j)}2} \langle \Op_\Gamma(a )
\phi_{j}, \phi_{k}\rangle \end{equation}
where the last identity holds only for real values of $r_j, r_k$.
We
henceforth denote by $V^t$
 the  operator on symbols, defined formally by
\begin{equation} \label{QEXACT}  \alpha^t (\Op(a)) =
\Op(V^t(a)). \end{equation} See Section \ref{PSI} for precise assumptions on $a$.
 We shall denote $V^t_\Gamma$ when we want to stress the fact that $V^t$ act on $\Gamma$-invariant symbols.

The Wigner distributions have been studied a lot in the theory of
quantum ergodicity and quantum mixing. In the context of $\Gamma
\backslash G$, the Wigner distributions of this article are
studied in \cite{Z,Z3,W,SV,SV2} as well as in \cite{AZ}. In this paper we introduce the
family of (off-diagonal) Patterson-Sullivan distributions. They
are also constructed from the eigenfunctions
 $\phi_{j}$. Take the boundary values   $T_{j, \nu_j}(db)$ of $\phi_{j}$ and  $ T_{k, -\nu_k}(db)$ of $\phi_{k}$.

\begin{maindefn} \label{introdefPS} $PS_{j, k, \nu_j,-\nu_k}(db', db, d\tau)$ is the $\Gamma$-invariant distribution on $B^{(2)}\times \R\sim G$ defined by
$$PS_{(j,  \nu_j),(k, -\nu_k)}(db',db, d\tau)
=\frac{T_{\nu_j}(db)\overline{T_{-\nu_k}}
(db^\prime)}{|b-b^\prime|^{1+\nu_j-\overline{\nu_k}}} e^{(\nu_j+\overline{\nu_k})\tau}d\tau.
$$ \end{maindefn}
We note that the Patterson-Sullivan distributions depend on the eigenfunctions $\phi_j, \phi_k$, but also on the choice of the spectral parameters $\nu_j, \nu_k$ (in contrast with the Wigner distributions, which depend only on the eigenfunctions); hence the notation $PS_{(j,  \nu_j),(k, -\nu_k)}$. In the sequel, we will in general use the shorter notation $PS_{ \nu_j,-\nu_k}$, although it is slightly abusive.
In Proposition \ref{PSGAMMA}, we check that the distributions $PS_{\nu_j, -\nu_k}$ are
(right)-$\Gamma$-invariant distributions on $ B^{(2)}\times \R\sim G$. Besides, since the geodesic flow reads
$$g^t(b', b, \tau)= (b', b, \tau+t),$$
they are eigendistributions
for the geodesic flow in the sense that
\begin{equation} \label{EIGFORM} g^{t }_\sharp
PS_{\nu_j,-\nu_k} = e^{- t (\nu_j+\overline{\nu_k})}
PS_{\nu_j,-\nu_k}=e^{it (r_k-r_j)}
PS_{\nu_j,-\nu_k}
\end{equation}
(the last identity holds only for real values of $r_j, r_k$).
As a result,
$PS_{\nu_j,- \nu_k}$ induces an eigendistribution
$PS^{\Gamma}_{\nu_j, -\nu_k}$ of the geodesic flow on $\Gamma
\backslash G = S \X$, defined by
\begin{equation} \label{PSGAMMAa} \int_{\Gamma \backslash G} a\,
dPS^{\Gamma}_{\nu_j,- \nu_k} = \int_G  (\chi a) \,dPS_{\nu_j, -\nu_k},
 \end{equation} for every smooth
$\Gamma$-invariant function $a$. Once again $\chi$ is a smooth
fundamental domain cutoff, see \S \ref{FUND}. When
$\nu_j+\overline{\nu_k}=0$, these Patterson-Sullivan distributions
are the family of {\em diagonal} Patterson-Sullivan distributions,
introduced in \cite{AZ} (these are invariant by the geodesic
flow). Recently, Hilgert and Schr\"oder have extended the
definition and properties of off-diagonal Patterson-Sullivan
distributions  to more general symmetric spaces \cite{HilSc,
SchDiss}.

It was pointed out in \cite{Z} that the distribution
$\epsilon_{\nu_j}(z, b)e^{\la z, b\ra}\Vol(dz)db\defi e^{(\frac12 + \nu_j) \langle z,
b \rangle } T_{\nu_j} (db) \Vol(dz)$ on $S\X$
is a joint eigendistribution of the
horocycle and geodesic flows, contained in the (spherical) irreducible representation of $G$ generated by $\phi_{j}$.
The $PS$-distributions are new objects, and it is illuminating to
express them in terms of these more familiar ones.

\begin{prop} \label{PSEINTRO}  The Patterson-Sullivan distributions are given by the (well-defined)
products $\left(\epsilon_{\nu_j} \cdot \iota \overline{\epsilon_{-\nu_k}} \right)$~:
$$PS_{\nu_j, -\nu_k} (db', db, d\tau)= \frac{2^{-(\nu_j-\overline{\nu_k})}}{2\pi}\left(\epsilon_{\nu_j} \cdot \iota \overline{\epsilon_{-\nu_k}} \right)(z, b) \,e^{\la z, b\ra}\Vol(dz)db, $$
with $(z, b)\simeq (b', b, \tau)$ and
where $\iota$ is the involution $(b', b, \tau)\mapsto (b, b', -\tau)$, corresponding to $(x, \xi) \mapsto (x, -\xi)$ on $S\D$ or to the action of the non-trivial element of the Weyl group, $g\mapsto gw$ on $G$, $w=  \left( \begin{array}{ll}0 & -1 \\ & \\
1 & 0 \end{array}\right)$.\end{prop}

\begin{rem}We note that the Wigner distributions are naturally defined on $S\X\times \R$, whereas the Patterson-Sullivan distributions were only defined on $S\X$. In order to relate both, we need to extend the latter to $S\X\times \R$.
We do so by identifying $S\X$ with $S\X\times\{\frac{r_j+r_k}2\}.$ In other words, we extend the $PS$-distributions to
$S\X\times \R$ by considering
$$PS^{\Gamma}_{\nu_j, -\nu_k}\otimes \delta_{\frac{r_j+r_k}2}$$
at least for real values of $r_j, r_k$.
Doing so, we must pay special attention to the case of low laplacian eigenvalues, when $r$ is imaginary. For imaginary $r_j, r_k$, the formular above will
be generalized to
$$PS^{\Gamma}_{\nu_j, -\nu_k}\otimes \delta_{\frac{\nu_j-\overline{\nu_k}}{2i}},$$
and our formula will hold for functions $a(z, b, r)$ that extend holomorphically to imaginary values of $r$ (corresponding to kernels decaying fast enough away from the diagonal).

\end{rem}

\subsection{ \label{s:introinter}Definition of the intertwining operator.}

\begin{maindefn} \label{LDEF} The intertwining operator  $$\lcal:
C^\infty_c (G \times \R) \to C(G \times \R)$$ is defined by
\begin{equation} \label{INT} \lcal a(g, R)=\int  (1+u^2)^{-(\frac12+iR)}a (ga_{\tau-\frac{\log(1+u^2)}2} n_u, r)
e^{-2i(R-r)\tau}
dr du d\tau .
\end{equation}  \end{maindefn}
\noindent Here, $n_u \in N$ is the one-parameter unipotent
subgroup whose right-orbits define the horocycle flow and  $a_t \in A$
is  the one parameter subgroup whose right-orbits define the
geodesic flow (see \S \ref{REP}).

Extend the geodesic flow to $G\times \R$ by the formula\begin{equation}\label{e:speed}G^t(g, r)=(g a_{rt}, r).\end{equation}

We will also consider the geodesic flow as an operator acting on functions,
 by composition~: for a function $a$ on $G$, we denote $g^t a\defi a\circ g^t$, and for a function on $G\times \R$, $G^ta\defi a\circ G^t$.

 \subsection{Statement of results}

The main result
of this paper is that we have the intertwining relation
\begin{equation}\label{e:inter} \lcal\circ V^t=G^t\circ \lcal.\end{equation}
We prove \eqref{e:inter} in several levels. At first, we prove that
the dual (or adjoint) equation holds if we apply both sides to the
Patterson-Sullivan distributions. At this level, we do not specify
the domain and range of the dual intertwining operator  $\lcal$ on
symbols. Indeed, the domain and mapping properties of $\lcal$ on
$\X$ are rather subtle, and we only introduce the domain  in \S
\ref{s:quotient} after first studying the  mapping properties of
$\lcal$ in \S \ref{s:UC} on the hyperbolic disc  (Theorem
\ref{t:statement}). The final result is Theorem
\ref{INTintrodual}.

We now state the first level of the result, whereby the
intertwining relation comes from an exact relation between the
Wigner and Patterson-Sullivan families of distributions. If $a$ is
a function on $G$, decaying fast enough, and $\nu\in\C$, define
$$L_\nu a(g)=\int_\R (1+u^2)^{-(\frac12+\nu)} a(gn_u)du.$$
If $a$ is a function on $G\times\C$, and $\nu\in\C$, define the
function $a_\nu$ on $G$ by $a_\nu(g)=a(g, \nu)$. The following
(somewhat imprecise) statement asserts that the adjoint of $\lcal$
maps Patterson-Sullivan distributions to Wigner distributions. It
is stated precisely in \S \ref{s:interGamma}.

\begin{maintheo}\label{t:WPS} Let $a=a(z, b, r)$ be a $\Gamma$-invariant function, with
$$a(z, b, r)=\sum_{\gamma\in\Gamma}\tilde a(\gamma\cdot z, \gamma\cdot b, r),$$
with $\tilde a$ satisfying adequate decay and smoothness properties. Then
we have $$W_{j, k}^\Gamma(a)=2^{1+\nu_j-\overline{\nu_k}} PS_{\nu_j, -\nu_k}(L_{-\overline{\nu_k}}\tilde a_{\nu_j})
=  PS_{\nu_j, -\nu_k}(\lcal \tilde a).$$
\end{maintheo}
Theorem \ref{t:WPS} is first proved on the hyperbolic disc (Section \ref{s:UC}) by
introducing analogues of the Wigner and Patterson-Sullivan families there, and then goes down to
 the quotient $\X$ (Section \ref{s:quotient}). The ``adequate decay and smoothness properties'' are
  described in \S \ref{s:interGamma} ($\tilde a\in\scal_\infty^{\, \infty}$).

 We now state the final result in a somewhat imprecise form. In \S \ref{s:interGamma}, we define
 Hilbert spaces
  $\hcal_W^*$ and $\hcal_{PS}^*$.  The Wigner distributions form an orthonormal
 of  $\hcal_W^*$, while the Patterson-Sullivan distributions form
 an orthonormal basis of $\hcal_{PS}^*$. Both are modelled on the
 Hilbert space of
 Hilbert-Schmidt pseudo-differential operators. The dual Hilbert
 spaces $\hcal_W,$ resp. $\hcal_{PS}$ are the spaces of symbols on
 which $\lcal$ acts as an intertwining operator. The space
 $\Pi\scal_{\omega}^{\infty}$ is defined in Definition
 \ref{PIINFTY}.

 \begin{maintheo}
 \label{INTintrodualINTRO} The interwining operator  $
\lcal_{\Gamma\sharp} : \hcal_{PS}^* \to \hcal_{W}^*$  is an
isometric isomorphism, and   $ \lcal_{\Gamma\sharp} $ sends
$PS_{\nu_j, -\nu_k}$ to $W_{j, k}. $ Dually, there exists a class
of automorphic symbols,  $a\in \Pi\scal_{\omega}^\infty$, so that
$$\lcal_\Gamma\circ V^t_\Gamma a=G_\Gamma^t\circ \lcal_\Gamma a,$$
as an equality between two elements of $\hcal_{PS}$.

\end{maintheo}

There are similar partial results for the wave flows $e^{it
\sqrt{-\Lap}}$ and the modified wave flow $e^{it \sqrt{-\Lap -
\frac{1}{4}}}$  (see \S \ref{WAVEFLOW}).

\subsection{Asymptotic equivalence of  Wigner and Patterson-Sullivan distributions}
In \cite{AZ}, it is proved that,  after suitable normalization,
the diagonal Wigner distributions  and  Patterson-Sullivan
distributions are asymptotically the same in the semi-classical
limit. The same is true for the off-diagonal elements:
\begin{maintheo} \label{mainintro2} Let  $a \in C^{\infty}(\Gamma
\backslash G)$.  Given a sequence of pairs $(\nu_{j_n}, \nu_{k_n})$ of spectral parameters with $-i\nu_{j_n} \to + \infty$
and $|\nu_{j_n} - \nu_{k_n}| \leq \tau_0$ for some $\tau_0 \geq 0$, we have  the asymptotic formula
$$\int_{S\X} a(g)W^\Gamma_{{j_n}, {k_n}} (dg)= 2^{1+\nu_{j_n}-\overline{\nu_{k_n}}}\left(\frac{\pi}{r_{k_n}}\right)^{1/2} e^{-\frac{i\pi}4}\int_{S\X} a(g) PS^\Gamma_{\nu_{j_n}, -\nu_{k_n}} (dg)+ O(\nu_{k_n}^{-1}).$$
\end{maintheo}
The proof is very similar to that in the diagonal case in
\cite{AZ}, starting from Theorem \ref{t:WPS}. Hence we only sketch
the key points in \S \ref{APPENDIX}.
 This result has been extended to more general symmetric
spaces by Hilgert and Schr\"oder \cite{HilSc}.

 In \cite{Z2} (see also
\cite{Z5}), it is shown that the off-diagonal Wigner distributions
$W_{{j_n}, {k_n}}$ with $j_n\not= k_n$ and with a limiting spectral gap $r_{j_n} -
r_{k_n}\To \tau_0$ tend to zero when the geodesic flow is mixing,
at least after the removal of a subsequence of spectral density
zero. It then follows from Theorem \ref{mainintro2} that:

\begin{maincor}  Take a sequence of pairs $(j_n, k_n)$, with $j_n \not= k_n$ and $r_{j_n} - r_{k_n}\To \tau_0$. Assume that this sequence has positive density, in the sense that
$$\liminf_{\lambda\to +\infty}\frac{\sharp\{n, |r_{j_n}|\leq \lambda\}}{\sharp\{ j, |r_j|\leq \lambda\}}>0.$$

Then there exists a subsequence of full density such that $ r_{k_n}^{-1/2}PS_{\nu_{j_n}, -\nu_{k_n}} \To 0$.
\end{maincor}

 \subsection{Relations to other work}
 The existence of
the intertwining operator is rather unexpected from the viewpoint
of microlocal analysis and quantum chaos, but is quite natural
from the viewpoint of automorphic distributions and invariant
triple products  \cite{BR,BR2,MS,D,SV,SV2}, where it may be
interpreted as intertwining the family of Wigner  triple products
$\ell^W(a, \phi_{j}, \phi_{k} ) = \langle \Op(a)
\phi_{j}, \phi_{k} \rangle$ and the family of
Patterson-Sullivan triple products $\ell^{PS} (a, \phi_{j},
\phi_{k}) = \langle a, d PS_{\nu_j, -\nu_k} \rangle.$ It follows
from general principles that there exist constants $C_{r_j, r_k}$
$\ell^{PS} (a, \phi_j, \phi_k) = C_{r_j, r_k} \ell^{W} (a,
\phi_{j}, \phi_{k})$ and essentially $\lcal$ is an integral
operator with matrix elements $C_{r_j, r_k}$. Explicit formula are
given in \cite{AZ2} relating the triple products when evaluated on
automorphic symbols of a fixed weight. This approach through
explicit formulae and representation theory is a (less global)
alternative to the study of $\lcal$.  But $\lcal$ might have an
independent interest in representation theory as the intertwining
operator between these two families of triple products.

The relations between Wigner and Patterson-Sullivan distributions,
and the exact formulae relating them in \cite{AZ2}, shed some
light on the limit formula for  quantum variances of Wigner
distributions proved by Luo-Sarnak \cite{LS} and Zhao \cite{Zh}.
The quantum variance for a zeroth order pseudo-differential
operator $A$ is defined as
\begin{equation} \label{diag} V_A(\lambda) \defi
\frac{1}{N(\lambda)} \sum_{j:  \lambda_j \leq \lambda} |\langle A
\phi_j, \phi_j\rangle - \int a_0 d\omega|^2, \;\;\;\;\;
(N(\lambda) = \#\{j: \lambda_j \leq \lambda\})
\end{equation}
where $ \int a_0 d\omega$ is the Liouville average of the
principal symbol $a_0$ of $A$. It was suggested by Feingold and
Peres \cite{FP} that the quantum variance should tend to $0$ the
following way~:
 $$V_A(\lambda)\sim \frac{B(a_0, a_0)}{\sqrt{\lambda}}.$$
 The bilinear form $B$ should be proportional to $\hat{\rho}_{a_{0}, a_0}(0)$, where
 \begin{equation} \label{CORFUN} \rho_{\phi,\psi}(t)=\int_{S\X} \phi( x)\psi(g^tx)d\omega(x)-\int \phi d\omega\int
\psi d\omega \end{equation} is the ``dynamical correlation
function'', and $\hat{\rho}_{\phi,\psi}(\tau) = \int_{-
\infty}^{\infty} e^{- i \tau t} \rho_{\phi,\psi}(t) dt$ is its
Fourier transform. In \cite{LS,Zh} a version of this conjecture
(with additional arithmetic factors) was proved for the basis of
Hecke eigenfunctions.

We postpone further discussion to \cite{AZ2}, but would like to
state a conjecture that connects Patterson-Sullivan distributions
to the quantum variances in \cite{LS,Zh}. Unlike Wigner
distributions, Patterson-Sullivan distributions are defined
independently on the choice of a quantization procedure. In view
of Theorem \ref{mainintro2}, it is natural to consider the
classical variances for the diagonal $PS$-distributions:
\begin{equation} \label{PSdiag} PS_a(\lambda) \defi
\frac{1}{N(\lambda)} \sum_{j:  \lambda_j \leq \lambda} |\langle a,
\widehat{PS}_{\nu_j, -\nu_j} \rangle  - \int a d\omega|^2.
\end{equation}
Here, $\widehat{PS}_{\nu_j, -\nu_j} \defi \frac{1}{\langle \1,
PS_{\nu_j, -\nu_j} \rangle_{S\X}} \; PS_{\nu_j, -\nu_j}$ are
normalized $PS$-distributions (see \cite{AZ}) so that the
statement is correct for constant functions.  Since they are sums
of $g^t$-invariant bilinear forms, they should  have a closer
relation to dynamical correlation functions than variances for
Wigner distributions.

\begin{conj} Let $\X$ be a compact hyperbolic surface. There exists a constant $C_{\X}$ such that  $$PS_{a}(\lambda) \sim
\frac{C_{\X} \;\;\hat{\rho}_{a_0, a_0}(0) }{\sqrt{\lambda}}.
$$
\end{conj}
We note that the quantum variance associated with
Patterson-Sullivan distributions may still depend on the choice of
a basis of eigenfunctions.

Finally, we point out a possibly  tenuous relation of our
intertwining problem to the one studied by Bismut on locally
symmetric spaces of non-compact type in Chapter 10 of \cite{Bis}.
On the infinitesimal level, we are intertwining the generator of
the geodesic flow to the operator $P$ taking a symbol $a$ to the
symbol of $[\Op(a), \Lap]$. As discussed in \cite{Z3} the latter
operator $P$  has the form $H^2 + 4 X_+^2 + i r H$, where $H^2 + 4
X_+^2$ is elliptic along the stable foliation and $ir H$ is the
semi-classical operator of order $2$ where $H$ generates the
geodesic flow. By comparision, Bismut's hypoelliptic laplacian
$L_b^X$ is essentially the weighted sum of the harmonic oscillator
on the fiber of $T^* \X$ and $b H$. We note that the $H$ terms are
identical if we set $b = r$, while the other terms are in a sense
orthogonal (Bismut's is vertical while ours is horizontal). But
both have the essential property that as the semi-classical
parameter $b = r \to \infty$, the operators converge to the
generator of the geodesic flow.  There is a possible
 parallel of our conjugation problem to  the conjugation between the hypoelliptic
laplacian and a certain elliptic operator in \cite{Bis} (Chapter
10). We encounter similar problems in defining the domain of
$\lcal$ and its inverse.

\vspace{.5cm}
{\bf Notational issue~:} In what follows we have to face the issue that we are sometimes using
 the {\em sesquilinear} pairing between two $L^2$ functions (or more generally, two elements of a
  complex Hilbert space), and sometimes the {\em bilinear} pairing between a distribution and a test function
  (more generally, an element of a vector space and a linear form). We will try to keep distinct notations to
   avoid confusion, denoting by $\la a, b\ra$ the scalar product of two $L^2$ functions (linear w.r.t. $a$,
    antilinear w.r.t. $b$), and by $T(a)$ the pairing between a distribution $T$ and a test function $a$.
     More generally, we shall use the bracket notation $\la ., .\ra$ {\em exclusively} for sesquilinear pairings.

If $L$ is a linear operator on a space endowed with a sesquilinear form, we shall denote by $L^\dagger$
 its adjoint in the hermitian sense, that is, $\la  a,L b\ra=\la L^\dagger a, b\ra$.

We will denote $L_\sharp$ the adjoint of $L$ in the usual sense of duality~: if $T$ is a linear form, then $L_\sharp T$
 is the linear form defined by $(L_\sharp T)(a)\defi T(La)$. Thus, $L\mapsto L_\sharp$ is linear whereas $L\mapsto L^\dagger$ is antilinear.

If $T$ is a distribution and $\Phi$ a diffeomorphism, we shall also denote by $\Phi_\sharp T$ the
 pushforward of $T$ by $\Phi$~: $(\Phi_\sharp T)(a)\defi  T(a\circ\Phi)$. The two notations should not interfere.

\section{\label{BACKGROUND} Coordinates on $S\D$.}

 \subsection{\label{REP}Dynamics and group theory of  $G = PSL(2, \R)$}

A set of  generators of the Lie algebra $sl(2, \R)$ is given by
$$H = \left( \begin{array}{ll} 1 & 0 \\ & \\ 0 & - 1 \end{array}
\right), \;\;\; X_+ = \left( \begin{array}{ll} 0 & 1 \\ & \\ 0 & 0
\end{array} \right), \;\; Y =  \left( \begin{array}{ll} 0 & -1 \\
&
\\1 & 0 \end{array} \right).$$  The subgroups they generate are denoted by $A, N, K$ respectively.  We also put
$X_- = \left( \begin{array}{ll} 0 & 0 \\ & \\ 1 & 0
\end{array} \right),$
and denote the associated subgroup by $\overline{ N}$. In the
identification $S\D\equiv PSL(2, \R)$ the geodesic flow $(g^t)_{t\in\R}$ is given
by the right action of the group $A$ of diagonal matrices with positive entries~:
$g\mapsto ga_t$ where
$a_t = \left( \begin{array}{ll} e^{t/2} & 0 \\ & \\
0 & e^{-t/2} \end{array} \right).$  The action of the horocycle
flow $(h^u)_{u\in\R}$ is defined by the right action of $N$, in
other words by $g\mapsto gn_u$ where
$n_u = \left( \begin{array}{ll} 1 & u \\ & \\
0 & 1\end{array} \right).$
We shall also denote $\bar n_u = \left( \begin{array}{ll} 1 & 0 \\ & \\
u & 1\end{array} \right).$

 \subsection{\label{AC} Adapted coordinates}
As explained in the introduction, the identification $G\sim (G/K)\times (G/NA)\sim (G/K)\times K$ leads to the coordinates $(z, b)$ (where $z\in \D$, $b\in G/NA\sim K\sim S^1$) to parametrize points in $G\sim S\D$. The identification $G\sim (G/K)\times (G/NA)$ is $G$-equivariant, and thus the action of $g$ on itself by left-multiplication reads $g \cdot (z, b)=(g \cdot z, g \cdot b)$, where on the first component $G$ acts by isometry on the symmetric space $G/K$, and on the second coordinate $G$ acts on the boundary $G/NA\sim K$.

We denote by $o$ the origin $eK$ in $G/K$, and by $ \langle z, b \rangle $ the
signed distance to $o$ of the horocycle through the points $z \in
\D, b \in B$. This notation follows \cite{He}, but we warn again that our normalization of the metric differs by a factor $2$ from Helgason's.

We also use the identification $G\equiv S
\D \equiv  B^{(2)}\times \R.
$
It is based on the identification of $B^{(2)}=B \times B \backslash
\Delta$ (where $\Delta \subset B \times B$ is the diagonal) with
the space of oriented geodesics of $\D$. To $(b', b)\in B \times B \setminus
\Delta$ there corresponds a geodesic $\gamma_{b', b}$ whose
forward endpoint at infinity equals $b$ and whose backward
endpoint equals $b'$. The choice of time parameter is defined so
that  $(b', b, 0)$ is the closest point $z_{b', b} $ to the origin $o$ on
$\gamma_{b', b}$, and $(b', b, t)$ denotes the point $t$ units from
$(b', b, 0)$ in signed distance towards $b$. We note that
$\langle z_{b, b^\prime},b\rangle=\langle z_{b,
b^\prime},b^\prime\rangle$.  We define
$g(b', b) \in PSU(1,1)$ to be the unique element satisfying
\begin{itemize}

\item[]  $g(b', b) \cdot  1 = b$,

\item[] $g(b',b) \cdot  (-1) = b',$

\item[] $g(b', b) \cdot o = z_{b', b}$,

\end{itemize}
where $1, -1, b, b'$ are points of the boundary $B=S^1$, seen as the unit circle in $\C$ in the disc model.
\noindent We thus identify $  B^{(2)}\times \R\simeq G $ by $(b', b, t) \mapsto g(b', b) a_t$.
In these coordinates, the action of $g\in G$ by left-multiplication is expressed by
\begin{equation}\label{e:gaction}g\cdot(b', b, t)=\left(g\cdot b', g \cdot b, t+\frac{\la g\cdot o, g\cdot b\ra- \la g\cdot o, g\cdot b'\ra}2\right).
\end{equation}

%Once $b', b$ are fixed, we write $z=(\tau,u)$ where $\tau$ measures
%arclength on $\gamma_{b',b}$ and $u$ measures arclength on
%horocycles centered at $b$. As above, we define the origin  $\tau
%= 0$ as the point
% $z_{b,b^\prime}$ on $\gamma_{b, b'}$ such that $\langle
%z_{b, b^\prime},b\rangle=\langle z_{b, b^\prime},b^\prime\rangle$.
%The  coordinates $(\tau, u)$ parametrizing $z$ are defined by
%$(z,b)=g(b',b)a_{\tau} n_u$. For any given $(b',b)$, the volume
%element of $z$ is $dV(z)=|d\tau du|$.

 We will need the following formula~:
 \begin{lem}
\label{IDEN}  $$\log \frac{|b - b'|}2 + \langle g(b',b)a_{\tau} n_u\cdot o, b
\rangle = \tau,$$
where $b, b'$ are seen as elements of $S^1\subset \C$, and $|b-b'|$ is their usual distance in $\C$. \end{lem}

\begin{proof} To prove this, we use the identity
$$\langle g
a_t n_u\cdot o, g \cdot 1 \rangle = \langle a_t n_u\cdot o, 1 \rangle +
\langle g \cdot o, g \cdot 1 \rangle = t + \langle g \cdot o, g
\cdot 1 \rangle  $$ to  reduce the lemma  to the claim that
\begin{equation} \label{CLAIM} \langle g(b',b) \cdot o, g(b',b) \cdot 1 \rangle = - \log \frac{|b
- b'|}2. \end{equation}  However, a basic identity gives
\begin{equation}\label{BID}  |g \beta - g \beta'|^2  e^{  \langle g \cdot o, g\cdot \beta \rangle +
\langle g \cdot o, g \cdot \beta' \rangle} = |\beta - \beta'|^2.
\end{equation}
If we let $\beta = 1, \beta' = -1$ (so that $g(b', b) 1 = b, g(b',
b) (-1) = b'$) and recall that $\langle g(b', b)\cdot o, b \rangle =
\langle g(b', b)\cdot o, b' \rangle$, then (\ref{BID}) implies
\begin{equation}\label{BID2}  4 = |b - b'|^2  e^{ [ \langle z_{b', b}, b \rangle +
\langle z_{b, b'}, b' \rangle]} = |b - b'|^2  e^{2  \langle z_{b',
b}, b \rangle},
\end{equation}
which completes the proof of (\ref{CLAIM}) and hence of the lemma.

\end{proof}

\subsection{Time reversal}

Time reversal is the map $\iota~:(x, \xi) \to (x, - \xi)$ on the
tangent bundle. In the coordinates $(b', b, t)$ it takes the
form, \begin{equation} \label{TBBT} \iota(b', b, t) =  (b, b', -t).
\end{equation} That is, it reverses the endpoints of the oriented geodesic
$\gamma_{b, b'}$ and preserves the point $z_{b,b'}$ closest to $o$. In the group theoretic picture, time reversal is given by the action of the non-trivial element of the Weyl group,
$g\mapsto gw$ on $G$, where $w=  \left( \begin{array}{ll}0 & -1 \\ & \\
1 & 0 \end{array}\right)$.

\subsection{A coordinate change\label{s:calc}}
The formulae below are useful at several places in the paper.
\begin{equation} \label{KANa} n_u = k_u a_{- \log (1 + u^2)}
\bar n_{f(u)}, \end{equation} where $f(u) = \frac{u}{1 + u^2}$ and
where
$$k_u = \begin{pmatrix} \frac{1}{\sqrt{1 + u^2}}
&
\frac{u}{\sqrt{1 + u^2}} \\ & \\
- \frac{u}{\sqrt{1 + u^2}} & \frac{1}{\sqrt{1 + u^2}}
\end{pmatrix} .$$

This comes from the explicit calculation,
\begin{equation} \label{KAN}  \begin{pmatrix} 1 &u \\ & \\
0 & 1 \end{pmatrix}  =    \begin{pmatrix} \frac{1}{\sqrt{1 + u^2}}
&
\frac{u}{\sqrt{1 + u^2}} \\ & \\
- \frac{u}{\sqrt{1 + u^2}} & \frac{1}{\sqrt{1 + u^2}}
\end{pmatrix} \begin{pmatrix} \frac{1}{ \sqrt{1 + u^2}}& 0 \\ & \\
0  &\sqrt{1 + u^2}\end{pmatrix}  \begin{pmatrix} 1  & 0 \\ & \\
 \frac{u}{1 + u^2}  & 1 \end{pmatrix}. \end{equation}

 This formula implies that the element $g=n_u a_t\in G$ corresponds to the endpoints $b=1$,
 $b'=k_u w\in K$, in other words, using the $S^1$ model,
 $$b'=e^{2i\theta}, \qquad\mbox{ where } e^{i\theta}=\frac{u}{\sqrt{1+u^2}}+\frac{i}{\sqrt{1+u^2}}.$$

By calculation, we find $|b'-b|^2=\frac{4}{1+u^2}$ and $db'=\frac1\pi \frac{du}{1+u^2}.$ These calculations also show that
 $$\la n_u a_t, 1\ra=t$$ and
 $$ \la n_u a_t, b'\ra=-t+\log(1+u^2).$$
 For $t=\frac{\log(1+u^2)}2$, we see that $\la n_u a_t, 1\ra= \la n_u a_t, b'\ra$, and thus
 $n_ua_{\frac{\log(1+u^2)}2}=g_{b', b}$ (with $b=1, b'=e^{2i\theta}$ as above).
It follows that $g=n_u a_t\in G$ has the coordinates $(b', b, \tau)=(e^{2i\theta}, 1, t-\frac{\log(1+u^2)}2)$.

\section{Harmonic analysis on the hyperbolic disc and its compact quotients\label{s:HA}}

 \subsection{\label{MEASURES} Poisson 1-form, Haar measure and Plancherel measure}
 We shall denote by $db$ the normalized Haar measure on $K$, identified with the boundary $B$
 or with $S^1$.
The Poisson 1-form is defined by
\begin{equation} \label{P1FORM} P(z, b) db = e^{\langle z, b \rangle} db.
\end{equation}
Using the identities
\begin{equation}
\label{HELGID} \langle g  \cdot  z, g  \cdot  b \rangle = \langle
z, b \rangle + \langle g \cdot o, g \cdot  b \rangle,
\end{equation}
and \begin{equation}\label{e:derivative} \frac{d}{db} g \cdot b = e^{- \langle  g
\cdot o, g \cdot b \rangle}, \end{equation} it follows that
\begin{equation} P(g  \cdot  z, g \cdot   b) d (g \cdot b) = P(z, b) db.
\end{equation}
Haar measure on $G$ is denoted $d g$. In terms of $z, b$
coordinates it is given by
\begin{equation} \label{HAAR} d g = P(z, b)\Vol(dz) db,
\end{equation}
where $\Vol(dz)$ is the hyperbolic area form. Under the identification $G\sim S\D$, the Haar measure on $G$ is the same as
Liouville measure on $S\D$.
In the $(b', b, t)$ coordinates, Haar measure reads as follows:

\begin{lem} \label{POISBB} Under the identifications $(b, b', t) \simeq g=g(b, b') a_t \simeq (z, b),  $ we have $$dg = P(z, b)\Vol(dz)db=4\pi\frac{db \otimes db'}{|b
- b'|^2} \otimes dt . $$
\end{lem}
The fact that the measure $\frac{db \otimes db'}{|b - b'|^2}
\otimes dt  $ is invariant under the action of $G$ follows from
formulae \eqref{e:gaction}, \eqref{BID}, \eqref{e:derivative}. We
leave aside the calculation of the normalization factor, which can
be done thanks to the formulae in \S \ref{s:calc}, but does not
play a very important role.

Non-Euclidean Fourier analysis is based on the family of
non-Euclidean ``plane waves'' $$e_{\nu, b}(z)\defi e^{(\frac12+\nu)\langle
z, b \rangle},$$
$\nu\in\C, b\in B$.
They are complex-valued  eigenfunctions of the
laplacian~:\begin{equation*}  \Lap e_{\nu, b}  = - \left(\frac14 -\nu^2\right) e_{\nu, b}.\end{equation*}  The $L^2$ spectral decomposition of the laplacian on $\D$ only requires the tempered spectrum,
that is, the case $\nu=ir$ where $r\in\R$ (corresponding to a laplacian eigenvalue $ \frac14 -\nu^2 \geq \frac14$).
The Helgason-Fourier transform of a function $f$ on $\D$ is defined by
$$\fcal f(b, r)=\int_{\D} e^{(\frac12-ir)\la z, b\ra}f(z)\Vol(dz),$$
$b\in B, r\in\R$. The Fourier transform automatically has the following symmetry property under $r\mapsto -r$~:
\begin{equation}\label{e:Fsym}\int_B \fcal f(b, r) e^{(\frac12+ir)\la z, b\ra}db =\int_B \fcal f(b, -r) e^{(\frac12-ir)\la z, b\ra}db,
\end{equation} for all $z\in\D$ and $r\in\R$.
Plancherel measure is the measure on $\R$ defined by
\begin{equation}\label{PLAN}  dp(r) = \frac{1}{2\pi}r \tanh( \pi r) dr,\end{equation}
and the Fourier inversion formula reads
$$f(z)=\frac12\int_\R \int_B \fcal f(b, r)e^{(\frac12+ir)\la z, b\ra}dp(r) db=\int_{\R_+} \int_B \fcal f(b, r)e^{(\frac12+ir)\la z, b\ra}dp(r) db,$$
see \cite{He}.
We have the Plancherel formula for $f\in L^2(\D)$, $\norm{f}_{L^2(\D, {\rm Vol})}
=\norm{\fcal f}_{L^2(B\times \R_+, db\times dp(r))}.$

 \subsection{Integral representation of eigenfunctions}
 We now consider Fourier analysis on the quotient $\X$ of $\D$ by a
discrete co-compact subgroup $\Gamma \subset G$.

\begin{thm}\label{Helga} (\cite{He}, Theorems 4.3 and 4.29; see also \cite{He2})  Let $\phi$ be an
 eigenfunction with exponential growth, for the
eigenvalue $\lambda=-\left(\frac14-\nu^2\right)\in \C$. Then there exists a
distribution $ T_{\nu, \phi} \in {\mathcal D}'(B)$ such that
$$\phi(z) = \int_B e^{(\frac12 + \nu) \langle z,
b \rangle } T_{\nu, \phi}(db),$$ for all $z\in\D$. The
distribution is unique if $\frac{1}{2} + \nu \not= 0, -1,
-2,\cdots$.
\end{thm}
The distribution $T_{\nu, \phi}$ is usually called the {\em boundary values} of $\phi$ (for the spectral parameter $\nu$), in analogy with the theory of boundary values of harmonic functions.
This theorem applies, in particular, to a $\Gamma$-invariant eigenfunction of the laplacian, since such a function is bounded. By uniqueness of $T_{\nu, \phi}$, we see that $\phi$ being $\Gamma$-invariant is equivalent to
\begin{equation}\label{e:conform}\gamma^{-1}_\sharp T_{\nu, \phi}(db)=e^{-(\frac12 +\nu)\la \gamma\cdot o, \gamma \cdot b\ra}T_{\nu,\phi}(db)
\end{equation}
 for $\gamma\in\Gamma$ and $b\in B$.

\subsection{\label{FUND} Fundamental domains and cutoffs for $\Gamma \backslash G$}

 We  denote by $\dcal$ a fundamental
domain for the action of $\Gamma$ on $\D=G/K$. We use
the same notation for the fundamental domain $\dcal$ lifted to $G$.

When dealing with integrals against irregular distributions, it is
convenient to replace the characteristic function of a fundamental
domain by a smooth (compactly supported) cutoff $\chi$ on $G$ satisfying
$$\Pi\chi=1,$$
where we define the periodization operator $\Pi$ by
$$\Pi\chi(g)=\sum_{\gamma\in\Gamma}\chi(\gamma g).$$
Existence of such functions $\chi$ is obvious. We will call such a function $\chi$ a {\em smooth fundamental cutoff} for the action of $\Gamma$ on $G$. When needed, we may assume that $\chi$ is a right-$K$-invariant functions, that is, $\chi(z, b)=\chi(z)$.

Let $\chi, \chi'$ be two smooth fundamental cutoffs.
We will use repeatedly the following~: if $T$ is a $\Gamma$-invariant distribution on $G$, then  $T(\chi f)=T(\chi' f)$, for any $f \in C^\infty(\Gamma \backslash G)$ (seen as a $\Gamma$-invariant function on $G$).
To see this, write
\begin{equation} T(\chi f)=T(\chi f .(\Pi\chi'))= T(\Pi(\chi f) .\chi')=T(f\chi').
\end{equation}

\section{\label{PSI} Pseudo-differential calculus on the
Poincar\'e disc}

 Throughout this article,   we use a special hyperbolic
calculus of pseudodifferential operators introduced in \cite{Z3}.
In the hyperbolic calculus,  a
complete symbol $a(z, b, r)$ ($(z, b)\in \D\times B, r\in\R$) is quantized by the operator $\Op(a)$ on $\D$
defined  by
$$\Op (a) e_{ir, b}(z) = a(z, b, r) e^{(\frac12 +ir)\langle z, b
\rangle}$$
for $z\in \D, b\in B$ and $r\in\R_+$.
By the non-Euclidean Fourier inversion formula, we define $\Op(a)$
on  $C_c^{\infty}(\D)$:
$$\Op(a)u(z)=\int_B \int_{\R_+} a(z, b, r)e^{(\frac12+ir)\langle z, b \rangle} \fcal u(b, r)
dp(r)db.$$ We recall that  the measure $d p(r)=\frac{1}{2\pi}r \tanh( \pi r) dr$  is the Plancherel measure for $G$ (\ref{PLAN}).  At the formal level, the
kernel of $\Op(a)$ is thus given by
\begin{equation} \label{KERNEL} K_a(z,w) =  \int_B \int_{\R_+} a(z, b, r) e^{( \frac12+ir )\langle z,
b \rangle}  e^{( \frac12- ir )\langle w, b\rangle} dp(r) db.
\end{equation}

Now assume that $a$ has the following symmetry w.r.t. the transformation $r\mapsto -r$:
\begin{equation}\label{e:Wsym} \int a(z, b, r)e^{(1/2+ir)\la z, b\ra}e^{(1/2-ir)\la w, b\ra}db
= \int a(z, b, -r)e^{(1/2-ir)\la z, b\ra}e^{(1/2+ir)\la w, b\ra}db
\end{equation}
for all $z, w\in \D$ and $r\in\R$.
It then follows from the Plancherel formula for the
hyperbolic Fourier transform that we can recover the symbol from
the kernel by
\begin{equation} \label{SYMBOL} a(z, b, r) = e^{-( \frac12+ir )\langle z, b
\rangle} \int_{\D} K_a(z,w) e^{(\frac12+ir) \langle w, b
\rangle} \Vol(dw)
\end{equation}
for all $r\in\R$. In this case, formula \eqref{KERNEL} holds with $\int_{\R_+}$ replaced by $\int_{\R_-}$.
We now discuss several particular classes of symbols $a$.

\subsection{Hilbert-Schmidt operators on $\D$ and $L^2$ symbols}
We recall that if ${\mathcal H}$ is a Hilbert space, the
algebra of Hilbert-Schmidt operators on ${\mathcal H}$ is the
algebra of  operators $A$ for which the trace $\Tr A A^\dagger$ is finite; it is endowed
with the inner product $\langle A, B
\rangle_{\hs}\defi \Tr A B^\dagger$. It is well known that the Hilbert-Schmidt
operators on ${\mathcal H}=L^2(M, d\nu)$ for any measure space form a Hilbert
space isomorphic to $L^2(M \times M, d\nu \times d\nu)$. In the case $M=\D$
 and $\nu=\Vol$, we will denote $\hs(\D)$ the space of Hilbert-Schmidt operators on $\D$.
We denote $L_W^2(G \times \R, dg \times dp(r))$ the space of functions in $L^2(G \times \R, dg \times dp(r))$
 that have the symmetry \eqref{e:Wsym} with respect to the Weyl group. We endow it with the norm
\begin{multline*} \norm{a}_{L_W^2}^2=\frac12 \int_{\D} \int_B \int_{\R} |a(z, b, r)|^2 P(z, b) \Vol(dz)
db dp(r)\\=\int_{\D} \int_B \int_{\R_+} |a(z, b, r)|^2 P(z, b) \Vol(dz)
db dp(r).
\end{multline*}

The following is a consequence of the Plancherel formula~:
\begin{prop} \label{HSD} The quantization map $a \to \Op(a)$
defines a unitary equivalence
$$L^2_W(G \times \R, dg \times dp) \simeq \hs(\D). $$ In other words
$$||\Op(a)||^2_{\hs(\D)
} =\frac12 \int_{\D} \int_B \int_{\R} |a(z, b, r)|^2 P(z, b) \Vol(dz)
db dp(r). $$
\end{prop}

It is then clear that the time evolution $V^t$ \eqref{QEXACT} defines a unitary operator on $L_W^2(G \times \R, dg \times dp)$.

\subsection{Schwartz class and associated symbols\label{s:Schw}}
Schwartz functions on $G$ were first defined by Harish-Chandra \cite{HC66}; the definition was extended to $G/K$ by Eguchi and his collaborators  \cite{Eg74, Eg79}.
Writing the hyperbolic disc as $G/K$, $f$ belongs to the Schwartz space $\ccal^p(G/K)$ (for $0<p\leq 2$) if and only if
$f$ is a function on $G$ which is right-$K$-invariant, and
$$\sup_{g\in G} \varphi_o(gK)^{-2/p}(1+ d(gK, o))^q |LRf(g)|<+\infty,$$
for any $q>0$, and for any differential operators $L, R$ on $G$ which are respectively left- and right-invariant. Here $\varphi_o$ stands for the spherical function on $G/K$, $\varphi_o(z)=\int e^{\frac12\la z, b\ra}db$. It satisfies $\varphi_o(z)\asymp  d(z, o)e^{-d(z, o)/2}$ as the hyperbolic distance $d(z, o)\To +\infty$.
Functions on $\ccal^p(G/K)$ are, in particular, in $L^p$ (they are sometimes called Schwartz functions of $L^p$-type).

The Fourier transforms of Schwartz functions of $L^p$-type were characterized by Eguchi \cite{Eg74, Eg79}~: letting $\epsilon=\epsilon(p)=\frac2p-1$,
$\fcal \left(\ccal^p(G/K)\right)$ coincides with the space $\ccal( B\times \R^{\epsilon})_W$ of functions $u$ on $ B\times \R$ such that\\
-- $u$ extends holomorphically to the strip $\R^\epsilon=\{|\Im m(r)|<\frac{\epsilon}2\}$ (this condition is empty for $\epsilon =0$);\\
-- on this strip (or, in the case $\epsilon =0$, on the real axis), we have a bound
\begin{equation}\label{e:schw}\sup_{(b, r)}(1+|r|)^q \left|\frac{\partial^\alpha}{\partial r^\alpha} D u(b, r)\right|<+\infty,\end{equation}
for all $q>0$, every integer $\alpha$, and every $K$-left-invariant differential operator $D$ acting on $B$ (here we use the identification $B\sim K$);\\
-- besides, $u$ must satisfy the symmetry \eqref{e:Fsym} (this symmetry condition with respect to the Weyl group is indicated by the subscript $_W$).

 It is clear from this characterization that the Schwartz space $\ccal^p(G/K)$
is stable under $e^{it\frac{\Lap}2}$.

We now define the space $\kcal^{p, q}(G/K\times G/K)$ (resp. $\kcal_{p, q}(G/K\times G/K)$, $\kcal^p_{\, q}(G/K\times G/K)$, $\kcal_p^{\, q}(G/K\times G/K)$) of kernels of operators sending $\ccal^p(G/K)$
continuously
to $\ccal^q(G/K)$ (resp. $(\ccal^p(G/K))'$
to $(\ccal^q(G/K))'$, $\ccal^p(G/K)$
to $(\ccal^q(G/K))'$, $(\ccal^p(G/K))'$
to $\ccal^q(G/K)$). We denote the corresponding symbol classes by $\scal^{p, q}(G/K\times B\times \R)_W$, $\scal_{p, q}(G/K\times B\times \R)_W$, $\scal^{p}_{\, q}(G/K\times B\times \R)_W$ and so on.

All these classes of operators are obviously stable under conjugation by $e^{it\frac{\Lap}2}$, and thus $V^t$ preserves the corresponding symbol classes.

We will in particular consider the space $\kcal_{\infty}^{\,
\infty}(G/K\times G/K)$ of ``smoothing'' operators, sending
$(\bigcap \ccal^p(G/K))'$ to $\bigcap \ccal^p(G/K)$.
 Corresponding symbols $a(z, b, r)$ are characterized by the fact that $a(z, b, r)e^{(\frac12+ir)\la z, b\ra}$ belongs to
$\bigcap_\epsilon \bigcap_p \ccal( B\times \R^{\epsilon};
\ccal^p(G/K))_W$ (i.e. functions $a(z, b, r)$ with the $\ccal(
B\times \R^{\epsilon})$-regularity
 in the $(b, r)$ variables, taking values in $\ccal^p(G/K)$). We will denote this space of
 ``smoothing'' symbols by \begin{equation}
 \label{SCAL} \scal_{\infty}^{\, \infty}: = \scal_{\infty}^{\, \infty}(G/K\times B\times \R)_W
 : = \bigcap_\epsilon \bigcap_p \ccal( B\times \R^{\epsilon};
\ccal^p(G/K))_W. \end{equation}

\subsection{$\Op(a)$ and $\Op_{\Gamma}(a)$}

A key point of the non-Euclidean pseudo-differential algebra is
that it is automatically left invariant \cite{Z3}. We say that a symbol $a$
is $\Gamma$-invariant if $a(\gamma \cdot z, \gamma  \cdot b, r) = a(z, b, r)$.  Denote $t_g$ the action of $g\in G$ on functions on $G/K$, defined by $t_g f(z)=f(g^{-1}z)$.
We recall from \cite{Z3} that $a$ being $\Gamma$-invariant is equivalent to having $[t_{\gamma}, \Op(a)] = 0$ for all $\gamma \in\Gamma$. This commutation relation is also equivalent to the fact that $K_a(\gamma  \cdot z,\gamma  \cdot w) =
 K_a(z,w)$.
In this case, one may view $\Op(a)$ in either
of two ways: as an operator on $C_c(\D)$ or as an operator
 on $\Gamma$-invariant functions. The operators differ in the
 domains they are given. When we need to emphasize that the action takes place on $\Gamma$-invariant functions we denote the operator by $\Op_{\Gamma}(a)$. More generally, if an operator $A$ defined on functions on $\D$ commutes with all the translations $t_{\gamma}$ ($\gamma\in\Gamma$), we shall denote by $A_\Gamma$ the same operator acting on $\Gamma$-invariant functions.

 By the decay properties of the spherical function $\varphi_o$, we see that $L^2(\X)$ can be continuously embedded in $(\ccal^p(G/K))'$ if $p\leq 1$. As a result, if the kernel $K(z, w)$ is $\Gamma$-invariant, and is such that $\chi(z)K(z, w)\in \kcal_{p}^{\,p}(G/K\times G/K)$ (where $\chi$ is our smooth fundamental cut-off \S \ref{FUND}), then $K$ defines naturally a bounded operator on the quotient, $L^2(\X)\To L^2(\X)$~: for $\phi\in L^2(\X)$, one can define $K\phi$ by the identity
 $$\la K\phi, \psi\ra_\X\defi \la \chi K\phi, \psi\ra_\D,$$
for all $\psi\in L^2(\X)$ . Besides, this definition does not depend on the choice of the fundamental cut-off $\chi$.

We can rephrase this in terms of symbols. Assume that $a(z, b, r)$ is a $\Gamma$-invariant symbol, and is of the form $a(z, b, r)=\sum_{\gamma\in\Gamma} \tilde a(\gamma \cdot z, \gamma  \cdot b, r)$ where $\tilde a\in \scal_{p}^{\, p}$ and $p\leq 1$ (this holds, for instance, if $\tilde a( z, b, r)\defi\chi(z)  a( z, b, r)$ belongs to $\scal_{p}^{\, p}$). Then we can define the bounded operator
$\Op_\Gamma(a):L^2(\X)\To L^2(\X)$ by
$$\la \Op_\Gamma(a)\phi, \psi\ra_\X\defi \la \Op(\tilde a)\phi, \psi\ra_\D,$$
for $\phi, \psi\in L^2(\X)$. Again, this definition does not depend on the choice of $\tilde a$.

We shall denote by $\Pi$ the periodization operator
$$\Pi \tilde a(z, b, r)=\sum_{\gamma\in\Gamma} \tilde a(\gamma \cdot z, \gamma  \cdot b,
r).$$ The class of symbols in Theorem \ref{INTintrodualINTRO} is
given in:

\begin{defn}\label{PIINFTY}  We denote by  $\Pi \scal_{\infty}^{\,\infty}$ the image of
$\scal_{\infty}^{\,\infty}$ under $\Pi$, where
$\scal_{\infty}^{\,\infty}$ is defined in (\ref{SCAL}). \end{defn}

 For $a\in \Pi
\scal_{\infty}^{\,\infty}$, we can use Helgason's integral
 representation theorem \ref{Helga} to have an alternative formula for the action of a pseudodifferential operator on the quotient:
\begin{equation} \label{OpaGamma} \Op_{\Gamma}(a ) \phi_{j}(z) =
\int_B a(z, b, -i\nu_j) e^{(\frac12 + \nu_j) \langle z, b \rangle } T_{\nu_j} (db). \end{equation}
This expression makes sense for all values $\nu_j\in \C$ since $a(z, b, r)$ has a holomorphic continuation to $r\in\C$.

There is a standard relation between the Schwartz kernel
$K_a(z,w)$ of $\Op(a)$ on $\D$ and the Schwartz kernel
$K_a^{\Gamma}(z,w)$  of $\Op_{\Gamma}(a)$ on $\X$, namely
 $$K_a^{\Gamma}(z,w) = \sum_{\gamma \in \Gamma} K_a(z, \gamma w),
 $$ where the kernel $K_a(z,w)$
on $\D$ is defined in (\ref{KERNEL}).    One has formally
 $$\int_{\X} K_a^{\Gamma}(z,w) f(w) \Vol(dw) = \int_{\D} K_a(z,w)
 f(w)  \Vol (dw)$$
 for any $\Gamma$-automorphic $f$ and $a$.
On the other hand, $K_a^{\Gamma}$  has a distributional
eigenfunction expansion \begin{equation} \label{EFNEXP}
K_a^{\Gamma}(z, w) = \sum_{j, k} \langle \Op_\Gamma(a) \phi_{j},
\phi_{k} \rangle \; \overline{\phi_{j}(z)} {\phi_{k}}(w).
\end{equation}

\section{Intertwining the geodesic flow and the Schr\"odinger group on the universal cover\label{s:UC}}
In this section, we prove the intertwining formula \eqref{e:inter} on the universal cover $\D$. This is done by defining analogues of Wigner and Patterson-Sullivan distributions on $\D$ and by finding an explicit relation between both.

\subsection{``Wigner distributions'' on $\D$.}% We now define Wigner distributions for the action of $\Op(a)$  on $\D$.

%In fact, as long as the kernel $K(z, w)$ has a well-defined Fourier transform with respect to $w$, we can
%define the complete symbol
%$$a(z, b,\lambda)=e^{-(\frac12+i\lambda)\la z, b\ra}\int K(z, w) e^{(\frac12+i\lambda)\la w, b\ra}\Vol%(dw),$$
%where the last term means the Fourier transform of $K$ with respect to $w$.
%We will apply this to $K_{(r, b), (r', b')}(z, w)=e^{(\frac12+ir)\langle
%z, b\rangle}e^{(\frac12+ir')\langle
%w, b'\rangle}.$
%The family $e_{(ir, b)}(z)= e^{(\frac12+ir)\langle
%z, b\rangle}$ ($r>0$) forms a generalized  ``orthonormal'' basis of eigenfunctions of
%$L^2(\D)$. So does the family ${e_{(w(ir), b)}}(z)= e^{(\frac12+w(ir))\langle
%z, b\rangle}$. These plane waves are temperate distributions on $\D$ and have distributional Fourier transforms.

\begin{defn}\label{DWIGNER} For $b, b'\in B$ and $\nu, \nu'\in i\R$,
the  Wigner distributions ${ W}_{(\nu, b), (\nu',
b')} \in \dcal'(S\D \times \R)$ are defined formally by:
$$\int_{S\D \times \R} a(z, \tilde{b}, r)  W_{(\nu, b), (\nu', b')}(dz, d\tilde{b}, dr)
= \langle \Op(a) e_{\nu, b},{e_{\nu', b'}}\rangle$$
for $a$ having the symmetry \eqref{e:Wsym}.
\end{defn}

For $b\in B$, we denote
$\delta_b(d\tilde{b}) $ the distribution density on $B$ corresponding to the Dirac mass at $b$, defined by $\int_B f(\tilde b)\delta_b(d\tilde{b})=f(b)$ for every smooth $f$.

\begin{prop} \label{WIGNER} We have:
 $ W_{(\nu, b), (\nu', b')}(dz, d\tilde{b},d r) = e_{\nu, b}(z)
\overline{{e_{\nu', b'}}(z)} \delta_b(d\tilde{b}) \delta_{-i\nu}(dr)\Vol(dz).$

\end{prop}

\begin{proof} If $\nu=ir$, it is immediate from the definitions that
$$ \langle \Op(a) e_{\nu, b},
e_{\nu', b'}) = \int_{\D} a(z, b, r) e_{ir, b}(z) \overline{ {e_{\nu', b'}}(z)} \Vol(dz).
 $$

\end{proof}

We  define the {\it Wigner transform}
of a function $a \in C_c^{\infty}(G\times\R)$ obeying the symmetry \eqref{e:Wsym} by
\begin{eqnarray*} \wcal: C_c^{\infty}(G\times\R ) &\to& L^2 \left(B \times i\R \times B \times i\R, db \otimes p(dr) \otimes db'\otimes
p(dr')  \right), \\ \wcal a(\nu, b, \nu', b') &=& W_{(\nu, b), (-\nu',
b')}(a). \end{eqnarray*}
Note the ``minus'' sign in front of $\nu'$. The following proposition proves completeness of the Wigner
distributions.

\begin{prop} \label{WIGNERINV} The Wigner transform extends to $L^2_W(G\times\R,  dg \times dp(r))$ as an isometry and  satisfies the inversion formula,
$$a(z, b, r) = e^{-(\frac12+ir) \langle z, b \rangle}\frac12 \int_B \int_{\R}
e^{(\frac12-ir') \langle z, b' \rangle}  \wcal a(ir, b,i r', b') db'
dp(r'). $$
\end{prop}

\begin{proof} For $r, r'\in\R$, the Wigner transform is given by
$$ \wcal a(ir, b, ir', b') = \int_{\D} a(z, b, r) e^{(\frac12+ir) \langle z, b \rangle}
e^{(\frac12+ir') \langle z, b' \rangle}\Vol(dz) ,$$
it is the Fourier transform of $a(z, b, r) e^{(\frac12+ir) \langle z, b \rangle}$ with respect to $z$, evalutated at $(b', -r')$.
The inversion formula and the isometry
  \begin{eqnarray}\label{e:isomHaarW}||a||_{L^2_W(G\times\R,  dg \times dp(r))} &=& ||\wcal a(ir, b, ir', b')||_{L^2(B \times
i\R_+ \times B \times i\R_+, db \otimes p(dr) \otimes db'\otimes
p(dr')  )} \\
 &=&\frac12 ||\wcal a(ir, b, ir', b')||_{L^2(B \times
i\R \times B \times i\R, db \otimes p(dr) \otimes db'\otimes
p(dr')  )}.
\end{eqnarray}
follow from the Plancherel and inversion formulae for $\fcal$.
\end{proof}

\subsection{ \label{PSD} Patterson-Sullivan  distributions on $\D$}

\begin{defn} \label{DPS}  For $\nu, \nu'\in i\R$, the Patterson-Sullivan distribution $PS_{(\nu,b), (-\nu', b')}\defi
PS_{e_{(\nu,b)}, e_{(-\nu', b')}}$ associated  to the  two eigenfunctions
$e_{(\nu, b)}(z)= e^{(\frac12+\nu)\langle z, b\rangle}$ and $e_{(-\nu',
b')}(z)= e^{(\frac12-\nu')\langle z, b'\rangle}$  is the
distribution on $S\D = B^{(2)} \times
\R  $ defined by
\begin{equation}\label{introdefPSERB}  PS_{e_{(\nu,b)}, e_{(-\nu', b')}}(d\tilde{b},d\tilde{b}^\prime,
d\tau) =  \frac{\delta_b(d\tilde{b})
\delta_{b^\prime}(d\tilde{b}')}{|\tilde{b}-\tilde{b}'|^{1+\nu-\overline{\nu'}}} e^{(\nu+\overline{\nu'})\tau}d\tau.
  \end{equation}
\end{defn}
\noindent We use the coordinates defined in \S \ref{AC}. We note that
$PS_{e_{(\nu,b)}, e_{(-\nu', b')}}$ is undefined if $b = b'$. We chose this somewhat awkward notation (writing $\overline{\nu'}$ instead of $-\nu'$) so that the definition can be straightforwardly extended to $\nu\in\C$ when we go to a compact quotient. For higher rank symmetric spaces, this formula was generalized by Schr\"oder \cite{SchDiss}. He pointed out the fact that $\nu'$ has to be replaced by $- w.\nu'$ (where $w$ is the {\em longest element} of the Weyl group) if one wants the ``diagonal''  ($\nu=\nu'\in i{\frak a}^*$) Patterson-Sullivan distributions to be $A$-invariant.

We now  prove an analogue of Proposition \ref{PSEINTRO} on the universal cover. Recall that $S\D$ is naturally endowed with the density $e^{\la z, b\ra}\Vol(dz)db$, corresponding to the Liouville measure on the unit tangent bundle, or to the Haar measure in the group theoretic picture $S\D=G$. In what follows we have to distinguish between {\em distributions} and {\em distribution densities} on a manifold, see \cite[Ch.\ VI]{Hor}. The choice of a preferred density allows to identify both.
On the boundary $B$ (endowed with the density $db$) we will denote $\delta_{b_o}(b)$ the distribution defined by the Dirac mass at a point $b_o$, and $\delta_{b_o}(db)=\delta_{b_o}(b)db$ the corresponding distribution density, defining the linear form $f\mapsto f(b_o)$ on $C^\infty(B)$.
We recall that distributions can be multiplied under certain assumptions on their wavefront sets \cite{Hor} Thm 8.2.10.

\begin{prop} \label{PSERB2}
On $S\D$, define the distribution $\epsilon_{\nu, b}(z, \tilde b)=e^{(-\frac12+\nu)\langle z, \tilde b\rangle}   \delta_b(\tilde b), $ corresponding to the distribution density  $\epsilon_{\nu, b}(z, \tilde b)e^{\la z, \tilde b\ra}\Vol(dz)d\tilde b=e^{(\frac12+\nu)\langle z,\tilde  b\rangle}  \delta_b(\tilde b)d\tilde b \Vol(dz) . $

We have
$$ \left( \epsilon_{\nu, b}. \iota \overline{\epsilon_{-\nu', b'}}\right)(z, \tilde b)e^{\la z, \tilde b\ra}\Vol(dz)d\tilde b= 2\pi.  2^{(\nu-\overline{\nu'})}PS_{e_{(\nu,b)}, e_{(-\nu', b')}}(dz, d\tilde b) $$
where $\iota$ denotes time-reversal. The product on the left-hand side is well-defined for $b\not= b'$.
 \end{prop}

 \begin{proof} Writing $ \epsilon_{\nu, b}$ in $(b, b', t)$ coordinates, we
 have
\begin{equation}\epsilon_{\nu, b}(\tilde b, \tilde b', t) = e^{(- \frac12+\nu)\langle g(\tilde{b}, \tilde{b}') a_t \cdot 0, \tilde{b} \rangle}
\delta_b(\tilde{b}). \end{equation} Its time reversal is thus
\begin{equation} \iota \epsilon_{\nu, b} (\tilde{b}, \tilde{b}', t) = e^{(- \frac12+\nu)\langle g(\tilde{b'}, \tilde{b}) a_{-t} \cdot 0, \tilde{b'} \rangle}
\delta_b(\tilde{b'}). \end{equation} By the identity of Lemma
\ref{IDEN} of \S \ref{AC}, we have \begin{equation} \langle
g(\tilde b',\tilde b)a_{t}, b \rangle = t - \log \frac{|\tilde b - \tilde b'|}2.
\end{equation}   Multiplying
the two distributions gives
\begin{equation}  \left( \epsilon_{\nu, b}. \iota \overline{\epsilon_{-\nu', b'}}\right)(\tilde b, \tilde b', t) = 2^{-1+\nu-\overline{\nu'}}\frac{e^{ t (\nu +\overline{\nu'}) }}{ |b -
b'|^{-1 + \nu - \overline{\nu'}}} \delta_b(\tilde{b}) \delta_{b'}(\tilde{b'}).
\end{equation}
Multiplying by $e^{\la z, b\ra}\Vol(dz)db=4\pi\frac{db\otimes db'}{|b-b'|^2}dt,$ we find
$$\left( \epsilon_{\nu, b}. \iota \overline{\epsilon_{-\nu', b'}}\right)(\tilde b, \tilde b', t)  4\pi\frac{db\otimes db'}{|b-b'|^2}dt= 2\pi. 2^{\nu-\overline{\nu'}} \frac{e^{ t  (\nu + \overline{\nu'})}}{ |b -
b'|^{1 + \nu-\overline{\nu'} }} \delta_b(\tilde{b}) \delta_{b'}(\tilde{b'})db\,db'\, dt.$$

 \end{proof}

 \begin{rem} \label{r:transport}As in \cite{AZ}, we observe that Patterson-Sullivan distributions are
eigendistributions of the geodesic flow. For a distribution density on $S\D$ we define $g^t_\sharp$ as the pushforward by $g^t$. For $g=(z, \tilde b)\in G$, we have $\epsilon_{\nu, b}(ga_t)=e^{(-\frac12+\nu)t}\epsilon_{ir, b}(g).$ This implies that
 $$g^t_\sharp(\epsilon_{\nu, b}(z, \tilde b)e^{\la z, \tilde b\ra}\Vol(dz)d\tilde b)=e^{(\frac12-\nu)t}
 \epsilon_{\nu, b}(z, \tilde b)e^{\la z, \tilde b\ra}\Vol(dz)d\tilde b.$$
 \end{rem}

 \begin{prop} For $\nu, \nu'\in i\R$, we have:
$$g^t_\sharp PS_{e_{(\nu,b)}, e_{(-\nu', b')}} = e^{-t(\nu + \overline{\nu'})} PS_{e_{(\nu,b)}, e_{(-\nu',
b')}}.$$
\end{prop}
The proof is immediate. If we extend $PS_{e_{(\nu,b)}, e_{(-\nu', b')}}$ to $S\D\times \R$ by taking
$PS_{e_{(\nu,b)}, e_{(-\nu', b')}}\otimes \delta_{\frac{\nu-\overline{\nu'}}{2i}}$, and if we extend the geodesic flow  to $S\D\times \R$ by letting
$$G^t(z, b, r)=(g^{rt}(z, b), r),$$ we have
$$G^t_\sharp\left( PS_{e_{(\nu,b)}, e_{(-\nu', b')}} \otimes \delta_{\frac{\nu-\overline{\nu'}}{2i}}\right)= e^{i \frac{(\nu^2- \overline{\nu'}^{2})t}2} PS_{e_{(\nu,b)}, e_{(-\nu', b')}} \otimes \delta_{\frac{\nu-\overline{\nu'}}{2i}}.$$
In other words, for $\nu=ir, \nu'=ir'$,
$$G^t_\sharp\left( PS_{e_{(ir,b)}, e_{(-ir', b')}} \otimes \delta_{\frac{r+r'}{2}}\right)= e^{-i \frac{(r^2-r^{'2})t}2} PS_{e_{(ir,b)}, e_{(-ir',
b')}}\otimes \delta_{\frac{r+r'}{2}}.$$

When working on compact quotients we will have to worry about a
possible extension of these formulae to the case of complex $r$,
and this is why we pay attention to write formulae that can be
adapted in a straightforward manner to $r\in\C$.

\subsection{Radon-Fourier transform along geodesics} $PS$-distributions are closely connected to the Radon transform along
geodesics.
 As reviewed in
\S \ref{AC}, the unit tangent bundle $S\D$ can be
identified with $B^{(2)}\times\R$: the set
$B^{(2)}$ represents the set of oriented
geodesics, and $\R$ gives the time parameter along geodesics. We
denote by $\gamma_{b', b}$ the oriented geodesic with endpoints
$b', b$.

\begin{defn} \label{RT} The geodesic  Fourier-Radon transform is defined by
$$\rcal: C_c(S \D) \to C_c( B^{(2)} \times \R), \;\;\;
\mbox{by}\;\; \rcal f(b', b, r) = \int_{\R} f(g(b', b)
a_{t}) e^{- i r t} dt.$$
\end{defn}
It is clear that $\rcal$ intertwines composition with
$g^t$ and multiplication by $e^{i r t}$, i.e.
\begin{equation} \label{RADINT}  \rcal ( f\circ g^t) (b', b, r) = e^{i r t}
\rcal f(b', b, r). \end{equation}
By the Fourier inversion formula, \begin{equation}
\label{RADINV} f(g(b', b) a_t) = \frac1{2\pi}\int_{\R} \rcal f(b', b, r)
e^{i r t} dr. \end{equation}

We call ``Patterson-Sullivan transform'' the pairing of the
family of $PS$-distributions with a test function.
\begin{defn} \label{PSTRDEF} The $PS$-transform is
defined as follows:

  $PS: C_c^{\infty}(G \times \R) \to
C^{\infty}(B^{(2)}\times i\R\times i\R)$
on $G$ by $$ PS  a(\nu, b, \nu', b') \defi PS_{(\nu, b), (-\nu',
b')} \left(a_{\frac{\nu-\overline{\nu'}}{2i}}\right) = \frac{1}{ |b - b'|^{1 + \nu-\overline{\nu'} }} \int_{\R} a\left(g(b',
b) a_{\tau}, \frac{\nu-\overline{\nu'}}{2i}\right) e^{(\nu+\overline{\nu'}) \tau} d \tau   $$
 \end{defn}

The $PS$-transform is  related to the Fourier-Radon transform as
follows:
\begin{equation} \label{PSRADON} PS a(\nu, b, \nu', b')
 =  \frac{1}{ |b - b'|^{1 + \nu-\overline{\nu'}} } \rcal a_{\frac{\nu-\overline{\nu'}}{2i}}(b', b,
i(\nu+\overline{\nu'})).
 \end{equation}

 Using the inversion formula for $\rcal$, one gets the inversion formula for the $PS$-transform~:

\begin{lem} \label{PSINVGa} The function $a$ is determined from
its $PS$-transform $PS  a(ir, b, ir', b')$ by
$$a(b', b, t, R) = \frac1\pi e^{2 iR t}|b - b'|^{1 +2iR}  \int_{\R}  PS  a(ir, b, i(2R-r), b') e^{-2i r t}
dr.
$$
\end{lem}

Using the Parseval identity for the 1-dimensional Fourier transform, as well as Lemma \ref{POISBB} expressing the Haar measure in the $(b', b, t)$-coordinates, one gets

\begin{lem}\label{HaarPS}
 $$\norm{a}^2_{L^2(G \times \R , dg \otimes dp)}=\frac1{\pi}\int_{b,b'\in B, r, r'\in\R}  |PS a(ir, b, ir', b')|^2 db db' \left(\frac{r+r'}2\right)  \tanh\left( \pi \frac{r+r'}2\right) dr dr'.$$
 \end{lem}

This formula can be compared to that obtained for the Wigner
transform \eqref{e:isomHaarW}. Note, however, that the $dr
dr'$-density is different in the two formulae. We also stress the
fact that we do not ask $a$ to have the symmetry \eqref{e:Wsym}
here.

\subsection{\label{LERINFTY}Operator sending the Patterson-Sullivan distributions to the Wigner distributions}
If $a$ is a function on $S\D\simeq G$, and $\nu\in\C$, we define the function $L_{\nu} a$ on $G$ by
$$L_{\nu} a(g)=\int_\R a(g n_u) (1+u^2)^{-(\frac12 +\nu)} du.$$

In this section, we prove the following~:
\begin{prop}\label{LPSW} Let $a\in C_c^\infty(G)$, $\nu, \nu'\in i\R$ and $(b', b)\in B^{(2)}$. Then $L_{-\overline{\nu'}}(a)\in C^\infty(S\D)$. Although $L_{-\overline{\nu'}}(a)$ is not compactly supported, the pairing $ PS_{(\nu, b)(-\nu', b')}\left(L_{-\overline{\nu'}}(a)\right)$ is well defined, and we have
$$  PS_{(\nu, b)(-\nu', b')}\left(L_{-\overline{\nu'}}(a)\right)= 2^{-(1+\nu-\overline{\nu'})}   W_{(\nu, b)(-\nu', b')}(a).$$
\end{prop}
%The fact that $\la L_{ir'}(a), PS_{(ir, b)(ir', b')}\ra$ is well defined if $a$ is compactly supported is easy to see.

A proof of this property by direct computation was given in \cite{AZ}, in the ``diagonal'' case $\nu=\nu'\in i\R$. The proof given there could be transposed to the general ``off-diagonal'' case. Here we give an alternative presentation based on Proposition \ref{PSERB2} and on the invariance properties of
the distributions $\epsilon_{\nu, b}$.

To check the formula,
 we use that $a_t n_u = n_{ue^{ t}} a_t$ and  the
$KAN$ (Iwasawa) decomposition,
\begin{equation}  n_u = k_u a_{- \log (1 + u^2)}
\bar n_{f(u)}, \end{equation}
as in \eqref{KANa}.

 We now consider the action of $L_{-\overline{\nu'}}$ on the product
 $ \left( \epsilon_{\nu, b}. \iota \overline{\epsilon_{-\nu', b'}}\right)$ (we note that the Poisson density $e^{\la z, \tilde b\ra}\Vol(dz)d\tilde b$ is $n_u$-invariant, so can be taken out of the integral defining  $L_{-\overline{\nu'}}$).

\begin{prop}  \label{Lactions}\begin{enumerate}

\item   $\epsilon_{\nu, b}$ is $n_u$-invariant;

\item $  L_{\nu} (\iota\epsilon_{\nu, b})(z, \tilde b)  = \pi e_{(\nu, b)}(z) $ (it does not depend on $\tilde b$, in other words, it is a right-$K$-invariant function);

\item we have $L_{-\overline{\nu'}} \left( \epsilon_{\nu, b}. \iota \overline{\epsilon_{-\nu', b'}}\right) =  \pi \epsilon_{\nu, b} \overline{e_{(-\nu', b')}} $.

\end{enumerate}

\end{prop}

\begin{proof}
(1) is obvious, and (3) follows from (2), so we just have to prove (2).

\begin{equation} \begin{array}{lll}  L_{\nu} (\iota\epsilon_{\nu, b})(g) & = &
  \int_{\R} (1 +
u^2)^{-(\frac{1}{2} +\nu)}    \iota\epsilon_{r, b}(g  n_u)
 du. \end{array}\end{equation}

We rewrite $n_u$ using  (\ref{KAN}). Then we:
 \begin{itemize}

 \item  remove the right $\overline{ N}$ factor since $  \iota\epsilon_{\nu, b} $
 is right-$\overline{ N}$-invariant;

 \item  replace the  $A$ factor inside
 $   \iota\epsilon_{\nu, b} $ by a factor $ e^{(-\frac12 +\nu)\log(1+u^2)}$ outside, since $\iota\epsilon_{\nu, b} $ is an
$A$-eigendistribution of eigenvalue $ = \frac{1}{2} - \nu$
which is evaluated for  $a_{- \log (1 + u^2)}$.

\item then change variables to $K$ with $\theta = \arctan u$.

\end{itemize}

We then have,
\begin{equation}
\begin{array}{lll}  L_{\nu} (\iota\epsilon_{\nu, b})(g) & = &
  \int_{\R} (1 +
u^2)^{-1 }   (\iota\epsilon_{\nu, b})(g) (g  k_u) du
  =\pi \int_K (\iota\epsilon_{\nu, b})(gk) dk \\
  &=& \pi \int_K \epsilon_{\nu, b}(gk) dk= \pi \int \epsilon_{\nu, b}(z, b') e^{\la z, b'\ra} db'
  =\pi  e^{(\frac12+\nu)\la z, b\ra} \end{array}\end{equation}
for $g=(z, \tilde b)$.
\end{proof}
Multiplying by the Poisson density, we finally get
$$L_{-\overline{\nu'}} \left( \epsilon_{\nu, b}. \iota \overline{\epsilon_{-\nu', b'}}\right)
e^{\la z, \tilde b\ra}\Vol(dz)d\tilde b
=  \pi  e_{(\nu, b)(z)} \overline{e_{(-\nu', b')}}(z) \Vol(dz)\delta_b(\tilde b)d\tilde b.$$
We recognize from Proposition \ref{WIGNER} the expression of the Wigner distribution, for a
test function $a$ that does not depend on the $r$-parameter. Comparing with Proposition \ref{PSERB2}, we obtain
Proposition \ref{LPSW}.

Proposition  \ref{LPSW} was proven for a function $a$ defined on $S\D$, in other words a function
on $S\D\times \R$ that does not depend on the last variable. In the sequel, we will apply
Proposition \ref{LPSW} to an arbitrary function on $S\D\times \R$, using it in the following form.
If $a$ is a function on $S\D\times \R \simeq G\times \R$, and $r\in\R$, we define the function $a_r$ on $S\D\simeq G$ by $a_r(g)=a(g, r)$. Proposition  \ref{LPSW} implies that
$$2^{-(1+\nu-\overline{\nu'})}   W_{(\nu, b)(-\nu', b')}(a)=  PS_{(\nu, b)(-\nu', b')}\left(L_{-\overline{\nu'}}(a_r)\right) ,$$
for $\nu=ir\in i\R$, and $\nu'\in i\R$.

\subsection{The operator $\lcal$}

Recall that we have extended the PS-distribution $PS_{(ir, b),(-ir', b')}$ ($r, r'\in\R$), originally defined
on $S\D$, to $S\D\times \R$, by tensoring it by $\delta_{\frac{r+r'}{2}}$ on the
$ \R$-variable.

We now look for an operator $\lcal$ that acts on functions (distributions) defined on $S\D\times \R$, with the property that $PS {\lcal a}(ir, b, ir', b')= \wcal a(ir, b, ir', b')$ ($r, r'\in\R$).
This means that we must have $PS {\lcal a}(ir, b, ir', b')=2^{1+ir+{ir'}} PS_{(ir, b)(-ir', b')}\left(L_{{ir'}}(a_r)\right)$.

By the PS-inversion formula (Lemma \ref{PSINVGa}), we have for all $(b', b)\in B^{(2)}$, $t\in\R$,
$R\in\R$,
\begin{multline}
{\lcal a}(b', b, t, R) = \frac{2^{1+2iR}}\pi e^{2 iR t}|b - b'|^{1 +2iR}  \int_{\R}  PS  (\lcal a)(ir, b, i(2R-r), b') e^{-2i r t}
dr\\
=   \frac{2^{1+2iR}}\pi    \int_{\R}  (1+u^2)^{-(\frac12+iR)}a_r\circ h^u(b', b, \tau)
e^{2i(R-r)(t-\tau-\frac{\log(1+u^2)}2)}
dr du d\tau\end{multline}

In other words, letting $g=(b', b, t)$,
\begin{equation}\label{e:deflcal}\lcal a(g, R)=  \frac{2^{1+2iR}}\pi\int  (1+u^2)^{-(\frac12+iR)}a (ga_{\tau-\frac{\log(1+u^2)}2} n_u, r)
e^{2i(r-R)\tau}
dr du d\tau.
\end{equation}

\subsection{Intertwining\label{s:statement}}
In this section, we prove that the  operator $\lcal$ intertwines
$V^t$ and $G^t$ on $\D$. We recall the Hilbert space $L^2_W$ in
Proposition \ref{WIGNERINV}.

\begin{maintheo} \label{t:statement}The intertwining operator  $\lcal$ is an isometry from
 $L_W^2(G \times \R, dg \times dp(r))$ to the space
 $\hcal_{PS}(\D)$  of functions such that
$$\frac14 \int  |PS a(ir, b, ir', b')|^2 db db' p(dr) p(dr')<+\infty.$$
 and
we have
\begin{equation}\label{p:inter}\lcal \circ V^t=G^{t}\circ\lcal,\end{equation}
where both sides are bounded operators from $L_W^2(G \times \R, dg
\times dp(r))$ to $\hcal_{PS}(\D)$.

\end{maintheo}

\begin{proof}

First, we consider the action of both sides on $a\in \scal_{\infty}^{\, \infty}$. We know that the space $\scal_{\infty}^{\, \infty} $
 is preserved by $V^t$. We shall first check that $\lcal a$ is a continuous function when $a\in \scal_{\infty}^{\, \infty}$, and that \eqref{p:inter} then holds as a pointwise equality between $\lcal (V^t a )$ and $G^{t}(\lcal a)$.

 For $g=(z, b)$, we see that
 \begin{equation} \label{e:pointwise}
  \begin{array}{lll} \lcal a(z,b, R)&=&  \frac{2^{1+2iR}}\pi\int  (1+u^2)^{-(\frac12+iR)}a\circ h^u\circ g^{\tau-\frac{\log(1+u^2)}2} (z,b, r)
e^{2i(r-R)\tau} dr du d\tau\\ && \\ &=& \frac{2^{1+2iR}}\pi\int
(1+u^2)^{-(\frac12+iR)}(I-\partial_r^2)^Na\circ h^u\circ
g^{\tau-\frac{\log(1+u^2)}2} (z,b, r) e^{2i(r-R)\tau} dr du
\frac{d\tau}{(1+4\tau^2)^N} ,
\end{array} \end{equation}

where the $r$-integration by parts is used to gain powers of $\tau$ and make the $\tau$-integral convergent. We also know from the definition of $\scal_\infty^{\,\infty}$ that
$|(I-\partial_r^2)^Na(z, b, r)|\leq  C_{N, M, x_0}(1+r^2)^{-M} e^{-M d(z, x_o)}$
for any $N, M>0$ and any given $x_o$.

If $z$ stays in a fixed compact set, denoting $(\tilde z, b)=h^u\circ g^{\tau-\frac{\log(1+u^2)}2} (z,b)$, one can check by hand that
$\exp d(\tilde z, x_o)\geq C_1(1+|u|)e^{|\tau|} -C_2$, with $C_1, C_2>0$. For instance for $(z, b)=e\in G$, we compute explicitly
$$a_{\tau-\frac{\log(1+u^2)}2} n_u= \begin{pmatrix} \frac{e^{\tau/2}}{(1+u^2)^{1/4}} & \frac{ue^{\tau/2}}{(1+u^2)^{1/4}}\\ & \\
0 & e^{-\tau/2}(1+u^2)^{1/4} \end{pmatrix} .$$
In the Poincar\'e upper plane model, identified with $PSL(2, \R)/K$, this element represents a unit
tangent vector based at
$$\tilde z= \frac{e^{\tau}}{(1+u^2)^{1/2}}i + \frac{ue^{\tau}}{(1+u^2)^{1/2}}.$$
The hyperbolic distance of this point to the origin $x_o=i$ is given by
$$\cosh d(\tilde z, i)=1+2\left[ \left(\frac{e^{\tau}}{(1+u^2)^{1/2}}-1\right)^2+\frac{u^2e^{2\tau}}{(1+u^2)} \right]4(1+u^2)^{1/2}e^{-\tau}\geq 8(1+u^2)^{1/2}e^{|\tau|}-7.$$
It follows that
\begin{equation}\label{e:decay}|(I-\partial_r^2)^Na(\tilde z, b, r)|\leq C(1+r^2)^{-K}(1+u^2)^{-K/2}e^{-N|\tau|}
\end{equation}
(with $C$ uniform as $z$ stays in a compact set), so that the integral \eqref{e:pointwise} does make sense, and defines a continuous function of the variables $(z,b, R)$. To keep this paper reasonably short, we do not investigate
the additional regularity properties of $\lcal a$.

Using \eqref{e:decay}, we also see that, for fixed $(z, b, R)$, we have $|(\lcal a)\circ g^t (z,b, R) |\leq
C_{z, b, R, M} e^{-M|t|}$ (for $M>0$ arbitrary), and it follows that $PS {\lcal a}(ir, b, ir', b')$ is perfectly well defined for any $r, r'\in\C$, $(b', b)\in B^{(2)}$. The Wigner transform  $\wcal a(ir, b, ir', b')$ is also perfectly well defined, and $\lcal$ has been constructed so that $PS {\lcal a}(ir, b, ir', b')=\wcal a(ir, b, ir', b')$.
We see that
\begin{multline*}PS({G^t\lcal a})(ir, b, ir', b')=e^{-i \frac{(r^2-r^{'2})t}2}PS({\lcal a})(ir, b, ir', b')
\\=e^{-i \frac{(r^2-r^{'2})t}2}\wcal a(ir, b, ir', b')=\wcal (V^t a)(ir, b, ir', b')=PS({\lcal V^t a})(ir, b, ir', b'),\end{multline*}
and inverting this formula we get that $G^t{\lcal a}={\lcal V^t a}$, for all $a\in\scal_{\infty}^{\, \infty} $ (and this equality holds pointwise).

We can now easily extend the intertwining formula to $a\in L_W^2(G
\times \R, dg \times dp(r))$. Using formula \eqref{e:isomHaarW},
we see (in a tautological way) that  $\lcal$ is an isometry from
$L_W^2(G \times \R, dg \times dp(r))$
to $\hcal_{PS}(\D)$, and we have $G^t\circ\lcal=\lcal\circ V^t$, where both sides are
   bounded operators from $L_W^2(G \times \R, dg \times dp(r))$ to $\hcal_{PS}(\D)$.

\end{proof}

Comparing with Lemma \ref{HaarPS}, we note that the norm on $\hcal_{PS}(\D)$ is not equivalent to the norm on $L_W^2(G \times \R, dg \times dp(r))$ (the $dr dr'$-densities differ by an unbounded factor).

 In the Section \ref{s:quotient}, we will mimick this construction to build two Hilbert spaces $\hcal_W(\X)$ and
 $\hcal_{PS}(\X)$ formed of $\Gamma$-invariant symbols, such that $\lcal$ sends   $\hcal_W(\X)$ isometrically to
 $\hcal_{PS}(\X)$, and such that the intertwining formula $G^t\circ\lcal=\lcal\circ V^t$ holds between these two spaces. On the quotient, $\hcal_W(\X)$ will be naturally identified with the space of Hilbert-Schmidt operators via the quantization procedure $\Op_\Gamma$, but will{\em{ not}} be equivalent to $L^2\left((\Gamma\backslash G)\times \R\right)$.

\begin{rem} Since $G^t$ preserves the variable $r$, Proposition \ref{p:inter} still holds if we modify the definition of $\lcal a(g, R)$ by a constant depending only on $R$. Thus, we have the choice
of a normalization factor for $\lcal$. We note that $G^t 1=1$ and $V^t 1=1$, so that it is quite
natural to renormalize $\lcal$ to have formally $\widehat\lcal 1=1$.
This means dividing $\lcal a(g, R)$ by
$$
\int  (1+u^2)^{-(\frac12+iR)} e^{-2i(R-r)\tau}
dr du d\tau=\pi \int  (1+u^2)^{-(\frac12+iR)}   du.$$
 The function $\mu_0(s)=\int_{-\infty}^{+\infty}(1+u^2)^{-s}du$ ($\Re e(s)>\frac12$)
  extends meromorphically to the whole complex plane by $\mu_0(s)=\frac{\Gamma(\frac12)\Gamma(s-\frac12)}{\Gamma(s)} $ (see p. 65-66 in \cite{He}).

 The renormalized $\widehat\lcal$ now satisfies
 $\widehat{PS}_{(ir, b), (-ir', b')}(\widehat\lcal a) = W_{(ir, b), (-ir', b')}(a)$
 if we define the normalized $\widehat{PS}$-distributions by
 \begin{equation} \label{PSNORMED} \widehat{PS}_{(ir, b), (-ir', b')}=  \pi \mu_0\left(\frac12 +i\frac{r+r'}2\right){PS}_{(ir, b), (-ir',
 b')}.\end{equation}

  \end{rem}

\begin{rem} \label{r:holomorph}When working on the quotient, we will need the following properties of $\lcal$, which
follow from its explicit expression \eqref{e:deflcal}.

First note that, if $a\in\scal_\infty^{\,\infty}$, then $\lcal a(g, R)$ has a holomorphic extension to $R\in\C$

Assume that $a(z, b, r)\in\scal_\infty^{\,\infty}$ satisfies in addition, for every $\epsilon>0$, $p$, all $q>0$, every integer $\alpha$, and every $K$-left-invariant differential operator $D$ acting on $B$, a bound of the form
\begin{equation}\label{e:schw2}\sup_{(r, b)} e^{q|r|} \left|\frac{\partial^\alpha}{\partial r^\alpha} D a(\bullet, b, r)\right|_{\ccal^p(G/K)}<+\infty,\end{equation}
in $\{|\Im m(r)|<\frac{\epsilon}2\}$. In other words, we strenghthen the definition of $\scal_\infty^{\,\infty}$ by asking that $a$ decay superexponential fast in $r$, instead of superpolynomially fast. We will denote by $\scal_\omega^{\,\infty}$ the space of such symbols.

Then, for any fixed $g\in G$ and $R\in\C$, the map $t\mapsto \lcal a(g a_t, R)$, originally defined for $t\in\R$, has a holomorphic extension to $t\in\C$. In particular, $(G^t\lcal a) (g, R)$ is well defined for $R\in\C$.

\end{rem}

\subsection{\label{WAVEFLOW} Remark on the wave flow}

Let us briefly discuss the case of the wave flow,
$e^{it\sqrt{-\Lap-1/4}}$ (or alternatively, $e^{it\sqrt{-\Lap}}$).
The corresponding quantum evolution is
\begin{equation} \label{EGORAUTw}  \beta^t (\Op(a)) =e^{-it\sqrt{-\Lap-1/4}}\Op(a) e^{it\sqrt{-\Lap-1/4}}=:\Op(U^t a). \end{equation}
The explicit expression of $U^t$ is given in \cite{Z, Z3}.

Since  $e^{it\sqrt{-\Lap-1/4}}e_{(ir, b)}=e^{itr}e_{(ir, b)}$ and
$e^{it\sqrt{-\Lap-1/4}}e_{(-ir, b)}=e^{itr}e_{(-ir, b)}$ for
$r>0$, we see that $U^t$ defines a unitary operator on
$L^2_W(G\times \R,  dg \times dp(r))$, and that $U^{t}_\sharp
W_{e_{(ir, b)}, e_{(-ir', b')}}= e^{it(r-r')}W_{e_{(ir, b)},
e_{(-ir', b')}}$ for $r, r'>0$.

We also have $g^t_\sharp PS_{e_{(ir, b)}, e_{(-ir', b')}}=
e^{it(r-r')}PS_{e_{(ir, b)}, e_{(-ir', b')}}$ (where $g^t$ is the
unit-speed geodesic flow).

It follows that $(\lcal\circ U^t)_\sharp PS_{e_{(ir, b)},
e_{(-ir', b')}}=(g^t\circ\lcal)_\sharp
 PS_{e_{(ir, b)}, e_{(-ir', b')}}$ for $r, r'>0$.  So in this sense,  $\lcal$ also intertwines the wave group and
the unit-speed geodesic flow. But because we restricted to
positive values of $r$, the result is not apriori as strong as for
the Schr\"odinger flow.

A further defect is that  $e^{it\sqrt{-\Lap-1/4}}$ does not
preserve
 the Schwartz spaces $\ccal^p(\D)$, because $\sqrt{.}$ is not a holomorphic
 function on the complex plane.  On a compact quotient, there
 is also a problem with the definition of $e^{it\sqrt{-\Lap-1/4}}$ for low
 eigenvalues~ (in particular, it is not unitary).  Some of these problems are circumvented by using
 $e^{it \sqrt{-\Lap}}$. But they explain why we prefer in this article to  work with the Schr\"odinger group $e^{it\Lap}$.
 We discuss the intertwining of the wave group and geodesic flow
 further in \cite{AZ2}, where the intertwining involves a
 modification of $\lcal.$

\subsection{Remark on semiclassical symbols and on an ``exact'' Egorov theorem\label{s:egorov}}
In the study of quantum chaos in the semiclassical r\'egime, one
often works with the flow $(e^{it\hbar\frac{\Lap}2})$, in the
limit $\hbar \To 0$. In this case, one considers semiclassical
symbols also depending on $\hbar>0$~: here, for instance, we could
define them as functions $a_\hbar\in \scal_\infty^{\,\infty}$,
having an asymptotic expansion
$$a_\hbar\sim \sum_{k=0}^{+\infty}\hbar^k a_k,$$
$a_k\in \scal_\infty^{\,\infty}$, the expansion being valid in all
the $\scal_\infty^{\,\infty}$--seminorms. In this setting, one
works with the operators $\Op_\hbar(a)\defi \Op(a(z, b, \hbar
r))$.

Introducing the operator $M_\hbar a(z, b, r)\defi  a(z, b, \hbar
r)$, it is natural in this context to introduce the notations
$\widehat\lcal_\hbar=M_\hbar^{-1}\circ \widehat\lcal\circ
M_\hbar$, and $V^t_\hbar= M_\hbar^{-1}\circ V^{t\hbar }\circ
M_\hbar$. Note that $M_\hbar^{-1}\circ G^{\hbar t}\circ M_\hbar=
G^t$.

In the semiclassical setting, the intertwining relation reads
$\widehat\lcal_\hbar\circ V^t_\hbar=G^t\circ \widehat\lcal_\hbar.$
 Using the stationary phase method, both sides of the equality $\widehat\lcal_\hbar\circ V^t_\hbar a_\hbar=G^t\circ \widehat\lcal_\hbar a_\hbar$ can be expanded into powers of $\hbar$, and on both sides the coefficient of $\hbar^0$ is $G^t a_0$~: this is an expression of the so-called ``Egorov theorem'', which says that $V^t_\hbar a_\hbar$ has an expansion starting with $G^t a_0+O(\hbar)$, combined with the fact that $\widehat\lcal_\hbar=I+O(\hbar)$.

 Our intertwining relation is intended to given an  ``exact'' form of the Egorov theorem~:
 if we know that $\widehat\lcal_\hbar$ is invertible, and have an explicit expression for its inverse, we can then define
 \begin{equation}\label{e:newquant}\widetilde{\Op_\hbar}(a)=\Op_\hbar(\widehat\lcal_\hbar^{-1}a),\end{equation}
 and this new quantization procedure will have the property that
 $$e^{-it\hbar\frac{\Lap}2}\widetilde{\Op_\hbar}(a)e^{it\hbar\frac{\Lap}2}=
 \widetilde{\Op_\hbar}(G^t a).$$
Such an exact intertwining relation is often called the ``exact''
Egorov property, and is so far only known in the euclidean case,
where the Weyl quantization $\Op^W_\hbar$ has the property
 that $e^{-it\hbar\frac{\Lap}2}\Op^W_\hbar(a)e^{it\hbar\frac{\Lap}2}=
\Op_\hbar(G^t a)$ (where $a$ is a function on $T^*\R^d=\R^d\times
\R^d$ with reasonable smoothness and decay properties, $\Lap$ is
the euclidean laplacian on $\R^d$, and $G^t$ is the euclidean
geodesic flow).

Being able to compute $\widehat\lcal^{-1}$ amounts to computing
$\lcal^{-1}$, and this can be done formally as follows~: we must
have $W_{(ir, b),(-ir', b')}(\lcal^{-1}a)=PS_{(ir, b),(-ir',
b')}(a)$, and we can recover the expression of $\lcal^{-1}a$ using
the inversion formula Proposition \ref{WIGNERINV}. We find
$$\lcal^{-1}a(z, b, r)=e^{-(\frac12+ir)\la z, b\ra}\int_{b'\in B, r'>0, \tau\in \R} \frac{a(b', b, \tau; \frac{r+r'}2)}{|b-b'|^{1+ir+ir'}}e^{(ir-ir')\tau}e^{(\frac12-ir')\la z, b'\ra} db' dp(r')d\tau.$$
To express this in group theoretic terms, let us consider the
special case $z=0, b=1$ (in the disc model), corresponding to
$g=e\in G$. Using the calculations of \S \ref{s:calc}, we find
$$\lcal^{-1}a(e, r)= \int  a\left(n_u a_{\tau+\frac{\log(1+u^2)}2}; \frac{r+r'}2\right)   \left(1+u^2\right)^{\frac{-1+ir+ir'}2}e^{(ir-ir')\tau} \frac{2^{-(1+ir+ir')}}{\pi}du dp(r') d\tau.$$
More generally, using $G$-equivariance of the formulae,
$$\lcal^{-1}a(g, r)= \int  a\left(gn_u a_{\tau+\frac{\log(1+u^2)}2}; \frac{r+r'}2\right)   \left(1+u^2\right)^{\frac{-1+ir+ir'}2}e^{(ir-ir')\tau} \frac{2^{-(1+ir+ir')}}{\pi}du dp(r') d\tau.$$

Using \eqref{e:newquant}, one defines a new quantization procedure
 satisfying the exact Egorov property. A drawback is that the regularity properties of $\lcal^{-1}$ are not well understood.
One can also note that, contrary to what is usually expected from
a quantization procedure, $\widetilde{\Op_\hbar}(a)$ is not the
multiplication by $a$ if $a=a(z)$ is a function on $G/K$.

  \section{Intertwining the geodesic flow and the Schr\"odinger group on a compact quotient \label{s:quotient}}

 \subsection{A remark on the spectrum of the geodesic flow and the Schr\"odinger flow}
 In this section, we investigate the meaning of the formula
 \begin{equation}\label{e:intert}\lcal \circ V^t=G^{t}\circ\lcal\end{equation}
 on a compact quotient. We already saw that this formula holds as a pointwise equality, when $a\in \scal_\infty^{\, \infty}$. If a formula such as  \eqref{e:intert} holds on a compact quotient, one would be tempted to infer that there is an explicit relation
 between the spectrum of the geodesic flow and the spectrum of the Schr\"odinger flow. On the other hand, one might object that the spectrum of $g^t$ on $L^2(\Gamma\backslash G)$
 is continuous (a manifestation of the ergodicity of the geodesic flow), whereas the spectrum of the quantum evolution $\alpha^t$ on the space of Hilbert-Schmidt operators is discrete. Of course, there is no contradiction at the end, since the spectrum depends on the space where an operator acts, and since $\lcal, \lcal^{-1}$ are unbounded operator.

In fact, the theory of resonances for the geodesic flow already describes its spectrum
on spaces other than $L^2$, and finds some Banach spaces of distributions where the
geodesic flows has some discrete spectrum (see the work of Ruelle and others, e.g.
\cite{BT, BKL, BL, FRS, GL, Liv, Ru87, Rugh92, Rugh96}). This approach was especially developed to describe
 the correlation spectrum of the geodesic flow for smooth (or H\"older) functions. As recalled
  in \cite{AZ}, the distributions $\epsilon_{\nu_j}$ are the generalized eigenfunctions arising
   in the resonance expansion of the geodesic flow. We must also note that the $\epsilon_{\nu_j}$
    are precisely the off-diagonal Patterson-Sullivan distributions $PS_{\nu_j, \frac{i}2}$, associated with the pairs of eigenfunctions
$\phi_{j}$ and the constant function $\equiv 1$, with spectral parameters, respectively, $\nu_j$ and $\frac{i}2$.
As to the other off-diagonal Patterson-Sullivan distributions $PS_{\nu_j, -\nu_k}$,
we saw they are also generalized eigenfunctions of the geodesic flow, but their spectral
 interpretation remains unclear~: they do not appear in the usual resonance theory.

In the next section, we shall construct a Hilbert space of distributions $\hcal_{PS}$,
 on which the geodesic flow acts, with dual eigenbasis the whole family $PS_{\nu_j, -\nu_k}$.
 In Definition \ref{introdefPS}, we introduced the Patterson-Sullivan distributions
  $PS_{\nu_j,-\nu_k}=PS_{(j,\nu_j),(k,-\nu_k)}$ associated with the pair of eigenfunctions $(\phi_j, \phi_k)$ and the
 choice of spectral parameters $(\nu_j, -\nu_k)$. On the quotient, the space $\hcal_{PS}$
  will be constructed by transposing the analysis done in \S \ref{s:UC}, replacing the distributions
$PS_{e_{(\nu,b)}, e_{(-\nu', b')}}$ by the family $PS_{\nu_j,-\nu_k}$.
A drawback is that elements of $\hcal_{PS}$ cannot be characterized in a simple way in terms
 of their regularity/growth properties (for instance, $\hcal_{PS}$ does not coincide
 with one of the known Sobolev spaces). However, we will give sufficient conditions for a distribution to belong to $\hcal_{PS}$ in Section \ref{HILBERT}.

 \begin{prop} \label{PSGAMMA} The distribution  $PS_{\nu_j,-\nu_k}$ is $\Gamma$-invariant. \end{prop}

\begin{proof} Recall that
$$PS_{\nu_j,-\nu_k}(db',db, d\tau)
=\frac{T_{\nu_j}(db)\overline{T_{-\nu_k}}
(db^\prime)}{|b-b^\prime|^{1+\nu_j-\overline{\nu_k}}} e^{(\nu_j+\overline{\nu_k})\tau}d\tau.
$$

We use the formulae of Section \ref{s:HA}, and the fact that
$\gamma^{-1}_\sharp T_{\nu_j}(db)=e^{-(\frac12 +\nu_j)\la
\gamma\cdot o, \gamma\cdot b\ra}T_{\nu_j}(db)$ for
$\gamma\in\Gamma$.
\begin{multline*} \gamma^{-1}_\sharp PS_{\nu_j,-\nu_k}(db',db, d\tau)=\frac{ \gamma^{-1}_\sharp T_{\nu_j}(db) \gamma^{-1}_\sharp\overline{T_{-\nu_k}}(b^\prime)}{|\gamma.b-\gamma.b^\prime|^{1+\nu_j-\overline{\nu_k}}}e^{(\nu_j+\overline{\nu_k})(\tau+\frac{\la \gamma o, \gamma b\ra- \la \gamma\cdot o, \gamma\cdot b'\ra}2)}d\tau   \\
=e^{-(\frac12+\nu_j)\langle \gamma\cdot o,
\gamma\cdot b\rangle}e^{-(\frac12-\overline{\nu_k})\langle \gamma\cdot o,
\gamma \cdot b^\prime\rangle}\frac{T_{\nu_j}(db)\overline{T_{-\nu_k}}(db')}{|b-b^\prime|^{1+\nu_j-\overline{\nu_k}}}e^{\frac12(1+\nu_j-\overline{\nu_k})
(\langle \gamma\cdot o, \gamma \cdot b\rangle+\langle \gamma o\cdot, \gamma\cdot
b^\prime\rangle)}
e^{(\nu_j+\overline{\nu_k})(\tau+\frac{\la \gamma \cdot o, \gamma \cdot b\ra- \la \gamma \cdot o, \gamma \cdot b'\ra}2)}d\tau \\= \frac{T_{\nu_j}(db)\overline{T_{-\nu_k}}(db')}{|b-b^\prime|^{1+\nu_j-\overline{\nu_k}}}
e^{(\nu_j+\overline{\nu_k})\tau }d\tau
= PS_{\nu_j,-\nu_k}(db',db, d\tau).
\end{multline*}

 \end{proof}

\subsection{Hilbert-Schmidt symbols}
We first have to remark that it does not make sense to try to generalize Proposition \eqref{p:inter}
to $\Gamma$-invariant symbols, in the form of a {\em pointwise} equality $\lcal_\Gamma \circ V_\Gamma^t a (z, b, r)=G_\Gamma^t\circ\lcal_\Gamma   a (z, b, r)$. Indeed, $V_\Gamma^t$ really acts on symbols {\em via} the definition
$e^{-it\frac{\Lap_\Gamma}2}\Op_\Gamma(a)e^{it\frac{\Lap_\Gamma}2}=\Op_\Gamma(V_\Gamma^t a)$, but on a compact quotient there is no inversion formula allowing to recover the symbol $a$ in a unique way from the operator $\Op_\Gamma(a)$. In fact, one can easily see that two different symbols on $S\X\times \R$ can yield the same operator on the quotient.

To put this another way, consider two symbols $\tilde a, \tilde b\in \scal_\infty^{\,\infty} $ which have the same periodization $a\defi\Pi \tilde a=\Pi \tilde b$. We can try to define
$V^t_\Gamma a$ by $V^t_\Gamma a\defi\Pi  V^t \tilde a$, but there is no reason why we should have
$\Pi  V^t \tilde a=\Pi  V^t \tilde b$ {\em pointwise}, in other words this definition of $V^t_\Gamma a$
will depend on the choice of $\tilde a, \tilde b$. On the other hand, $\Op_\Gamma(V^t_\Gamma a)$
does not depend on any choice, and coincides with $e^{-it\frac{\Lap_\Gamma}2}\Op_\Gamma(a)e^{it\frac{\Lap_\Gamma}2}$. This explains why we shall not try to prove that $\lcal_\Gamma \circ V_\Gamma^t a (z, b, r)=G_\Gamma^t\circ\lcal_\Gamma   a (z, b, r)$ pointwise -- one needs in some sense to quotient out by the symbols $a$ such that $\Op_\Gamma(a)=0$.

The aim of the section is to present a (simple) construction of two Hilbert spaces $\hcal_W$ and $\hcal_{PS}$, having respectively the families $(W_{\nu_j, -\nu_k})$ and $(PS_{\nu_j, -\nu_k})$ as dual orthonormal bases, with the following properties~: $V^{t}$ acts unitarily on $\hcal_W$, $G^{t}$
acts unitarily on $\hcal_{PS}$, $\lcal$ sends $\hcal_{W}$ isometrically to $\hcal_{PS}$, and the intertwining relation \eqref{e:intert} holds on these spaces. This constructions mimicks what was done at the end of \S \ref{s:statement}, with the additional difficulty just mentioned, of having to
quotient out by the symbols $a$ such that $\Op_\Gamma(a)=0$.

 We start with the Hilbert
space $HS(\X)\simeq L^2(\X \times
\X)$ of
  Hilbert-Schmidt
operators on  the compact quotient $\X=\Gamma \backslash G/K$. On this space, the quantum evolution $\alpha^t$
has the orthonormal spectral expansion
\begin{equation}\label{SPECALPHA}  \alpha^t = \sum_{j, k} e^{i t\frac{(\nu^2_j - \overline{\nu_k}^2)}2} \left( \phi_{j}\otimes \phi_{ k}^*\right)\otimes \left( \phi_{j}\otimes \phi_{
k}^*\right)^* \end{equation} (here, $\nu_k^2\in\R$, but we wrote $\overline{\nu_k}^2$ to underline the fact that the dependence of the expression is antiholomorphic w.r.t. $\phi_k, \nu_k$).
The Hilbert-Schmidt norm is defined by $\Vert A\Vert^2_{HS(\X)}=\Tr(A A^\dagger)$, associated with the scalar product $\la A, B\ra_{HS(\X)}=\Tr (AB^\dagger)$.
Starting with a $\Gamma$-invariant symbol $a$ belonging to $\Pi\scal_\infty^{\,\infty}$, we obtain a Hilbert-Schmidt operator $\Op_\Gamma(a)\in HS(\X)$, with norm
\begin{eqnarray*}\Vert\Op_\Gamma(a) \Vert^2_{HS(\X)}&=&\Tr_{L^2(\X)}  \Op_\Gamma (a) \Op_\Gamma
(a)^\dagger\\
=\sum_{j, k} |\Tr  \Op_{\Gamma}(a)\phi_{j}\otimes  \phi_{k}^*|^2
&=&\sum_{j, k} |W^\Gamma_{j, k}(a)|^2.
\end{eqnarray*}

This suggests to define the Hilbert space $\hcal_W$ as as follows:

\begin{defn} \label{HCALWDEF}
 $\hcal_W$  is the completion of the symbol space $\Pi\scal_\infty^{\,\infty}$
with respect to the norm
$$\Vert a\Vert^2_W=\sum |W^\Gamma_{j, k}(a)|^2.$$
The scalar product on $\hcal_W$ is defined by
$$\la a, b\ra_W=\sum W^\Gamma_{j, k}(a) \overline{W^\Gamma_{j, k}(b)}.$$
\end{defn}

Here we have to stress two important facts~:
 \begin{itemize}
 \item The norm $\Vert .\Vert_W$ is actually a seminorm. Again, this comes from the fact that we can have $\Op_\Gamma(a)=0$ although $a\not=0$. The reader should not be too surprised by this fact, which already occurs in the euclidean case when one wants to study the Weyl quantization $\Op^W(a)$
 of a symbol $a(x, \xi)$ (where $(x, \xi)\in T^*\R^d$), $2\pi\Z^d$-periodic in the $x$-variable. If $a\not=0$, then $\Op^W(a)$ defines a non-vanishing operator on $L^2(\R^d)$, but the periodization of this operator may vanish when acting on the torus $L^2(\R^d/2\pi\Z^d)$. Actually, this happens if $a(x, \xi)$ vanishes when $\xi$ is a half-integer. In the hyperbolic setting, a similar phenomenon occurs. Although the operator $\Op(a)$ is non-zero on $L^2(\D)$, its periodization $\Op_\Gamma(a)$ can vanish when acting on $L^2(\X)$. The difficulty is that there no easy characterization of the symbols $a$ such that $\Op_\Gamma(a)=0$.
 \item Another related issue is the following. On the universal cover $\D$, we have seen that the Hilbert-Schmidt norm of $\Op(a)$ coincides with the $L^2$-norm of $a$, seen as a function on $S\D\times \R$ endowed with the measure $e^{\la z, b\ra}dz db\times dp(r)$.
 This is no longer true on the quotient. In other words, for a $\Gamma$-invariant symbol $a$, $\Vert a\Vert_W$ is not the $L^2$-norm of $a$ on $S\X\times \R$ (see \S \ref{HSVSW}). Again, the same phenomenon already occurs in the euclidean case.
 \end{itemize}

 By definition,  $W^\Gamma_{j, k}$ is a bounded linear functional on
$\hcal_W$ for all $j, k$, in other words $W^\Gamma_{j,
k} \in \hcal_W^*,$
where
 $\hcal_W^*$ is the dual Hilbert space to $\hcal_W$. The Riesz theorem endows $\hcal_W^*$ with a dual inner product, and the $W_{j, k}$ form an orthonormal basis of
 $\hcal_W^*$.
For $b\in \hcal_W$, the series
\begin{equation}\label{VT=} \begin{array}{lll}   \sum_{j, k}
\overline{W^\Gamma_{j, k}(b) } W^\Gamma_{j,
k}
\end{array}\end{equation}
converges in $\hcal_W^*$.
In fact, the operator
$\sum_{j, k}
  \overline{ W^\Gamma_{j,
k} (\bullet)}  W^\Gamma_{j,
k}: \hcal_W \to \hcal_W^*
$
is just the standard (antilinear, unitary) isomorphism  $b \to
\langle \bullet,b \rangle_W$  from $\hcal_W \to \hcal_W^*$.

As described earlier, we can define $V_\Gamma^t$ acting on $\Pi\scal_\infty^{\,\infty}$ by $V_\Gamma^ta=\Pi V^t \tilde a$, if $a=\Pi \tilde a$ and $\tilde a \in \scal_\infty^{\,\infty}$. This
definition depends on the choice of $\tilde a$, however $W^\Gamma_{j, k}( V_\Gamma^t a)$ does not.
We then note that the evolution $V_\Gamma^t$ can be extended to $\hcal_W$, and is obviously unitary, since
we have $W^\Gamma_{j, k}( V_\Gamma^t a)= e^{it\frac{\nu_j^2-\overline{\nu_k}^2}2}W^\Gamma_{j, k}(a).$
The Wigner distributions form an orthonormal basis of eigenfunctions of the adjoint $V_{\Gamma\sharp}^{t}$ in  $ \hcal_W^*$~: we have $V_{\Gamma\sharp}^{t}(W^\Gamma_{j,
k})= e^{it\frac{\nu_j^2-\overline{\nu_k}^2}2} W^\Gamma_{j,k}.$

It is difficult to find a full characterizations of elements in $\hcal_W$ in terms of usual Sobolev spaces. However, in \S \ref{HILBERT},
we shall give a sufficient condition for a function to belong to of $\hcal_W$, in terms of its regularity and decay rate at infinity (in the $r$ variable).

Carrying this forward one step, we define another Hilbert space $\hcal_{PS}$. We consider the space $\cap_\epsilon \ccal(\R_\epsilon, C_c^\infty(G))$ of symbols $ \tilde a (g, r)$ which are holomorphic in $r\in\C$, taking values in the space of smooth compactly supported functions of $(z, b)$ (here the extension to $r\in\C$ is needed to deal with the case of low eigenvalues of the laplacian); and such that, for every integer $\alpha$ and any $k,q>0$, for any $\epsilon$, we have
 $$\sup_{r\in \R_\epsilon}\left\Vert (1+|r|)^q \frac{\partial^\alpha}{\partial r^\alpha} a\right\Vert_{C^k_{z, b}} <+\infty.$$

\begin{defn} \label{HCALPSDEF}
 We define $\hcal_{PS}$ as the closure of $\Pi \left(\cap_\epsilon \ccal(\R_\epsilon, \ccal_c^\infty(G))\right)$ under the ``scalar product''
 \begin{equation} \label{PSIP} \langle f, g \rangle_{PS} \defi \sum_{j, k} PS^\Gamma_{\nu_j,
-\nu_k}(f) \overline{PS^\Gamma_{\nu_j, -\nu_k} (g)}.
\end{equation}
\end{defn}
An alternative which is closer to  Definition \ref{HCALWDEF} would
be to use the normalized Patterson-Sullivan distributions
(\ref{PSNORMED}). Here we have to choose a value of the spectral
parameter $\nu_j$ for each $j$, and we make the standard (but
somehow arbitrary) choice~: $\nu_j\in[\frac12, 1]\cup i\R_+$. The
corresponding (semi)norm will be denoted $\Vert.\Vert_{PS}$. We
then have $PS^\Gamma_{\nu_j, -\nu_k} \in \hcal_{PS}^*,$ and
$\sum_{j, k} \overline{PS^\Gamma_{\nu_j, -\nu_k} (\bullet)}
PS^\Gamma_{\nu_j, -\nu_k}: \hcal_{PS} \to \hcal_{PS}^*$ is the
standard (antilinear, unitary) isomorphism from $\hcal_{PS} \to
\hcal_{PS}^*$. The evolution $G^t_\Gamma$ can be extended to a
unitary operator on $\hcal_{PS}$, or to $\hcal_{PS}^*$ by duality.

\subsection{Proof of Theorem \ref{INTintrodualINTRO}  \label{s:interGamma}}

We first state a more complete and precise version of the theorem.
The Hilbert spaces $\hcal_W, \hcal_{PS}$ are discussed in more
detail in \S \ref{HILBERT}.

\begin{maintheo}
 \label{INTintrodual} We have:

 \begin{enumerate}

 \item $V_\Gamma^t$ is a unitary operator on $\hcal_W$ and
 on $\hcal_W^*$. The Wigner distributions form an orthonormal
 basis of eigenvectors of $V_{\Gamma\sharp}^{t}$ in $\hcal_W^*$.

\item The maps  $\lcal_\Gamma: \hcal_{W} \to \hcal_{PS}, $ and $
\lcal_{\Gamma\sharp} : \hcal_{PS}^* \to \hcal_{W}^*$  are
isometric isomorphisms;  $ \lcal_{\Gamma\sharp} $ sends
$PS_{\nu_j, -\nu_k}$ to $W_{j, k}. $ For $a\in
\Pi\scal_{\omega}^\infty$, we have
$$\lcal_\Gamma\circ V^t_\Gamma a=G_\Gamma^t\circ \lcal_\Gamma a,$$
as an equality between two elements of $\hcal_{PS}$.

\end{enumerate}

\end{maintheo}

\begin{proof}
  The intertwining
relation on the quotient follows from the intertwining on the
universal cover, and from the relations
$$W^\Gamma_{j, k}(a)=W_{j, k}(\tilde a)=\int W_{e_{(\nu_j, b)}, e_{(-\nu_k, b')}}(\tilde a)dT_{\nu_j}(db)\overline{dT_{-\nu_k}}(db'),$$
\begin{equation}\label{e:PSquotient}PS^\Gamma_{\nu_j, -\nu_k}(a)=PS_{\nu_j, -\nu_k}(\tilde a)=\int PS_{e_{(\nu_j, b)}, e_{(-\nu_k, b')}}(\tilde a)dT_{\nu_j}(db)\overline{dT_{-\nu_k}}(db'),
\end{equation}
valid for any $\Gamma$-periodic $a$, and any $\tilde a$ such that $a=\Pi \tilde a$,  provided
$\tilde a$ has the sufficient smoothness and decay so that these expressions are well defined.

Take $a\in \Pi\scal_{\infty}^{\,\infty} $,
that is, $a=\Pi \tilde a$ with $\tilde a\in \scal_{\infty}^\infty$. We know from the calculations in \S \ref{s:UC} that
\begin{eqnarray*}
PS_{\nu_j, -\nu_k}(\lcal \tilde a)=
\int PS_{e_{(\nu_j, b)}, e_{(-\nu_k, b')}}(\lcal \tilde a)dT_{\nu_j}(db)\overline{dT_{-\nu_k}}(db')
\\=\int W_{e_{(\nu_j, b)}, e_{(-\nu_k, b')}}(\tilde a)dT_{\nu_j}(db)\overline{dT_{-\nu_k}}(db')
= W^\Gamma_{j, k}(a).
\end{eqnarray*}
Theorem \ref{t:WPS} is proved this way, and we also get a proof of Proposition \ref{PSEINTRO}, by combining Proposition \ref{PSERB2} on the universal cover, with formula \eqref{e:PSquotient}.

Now, take $a\in \Pi\scal_{\omega}^{\,\infty} $,
%ici \Pi\scal_{\infty}^\infty suffirait sans doute, si c'est le cas, le dire au thm 4
that is, $a=\Pi \tilde a$ with $\tilde a\in \scal_{\omega}^\infty$ (see Remark \ref{r:holomorph})
We know, from Remark \ref{r:holomorph}, that $G^t \lcal \tilde a(g, R)$ is well-defined for all $R\in\C$, and we have
\begin{eqnarray*}
PS_{\nu_j, -\nu_k}(G^t \lcal \tilde a)&=& e^{it\frac{\nu_j^2-\overline{\nu_k}^2}2}PS_{\nu_j, -\nu_k}( \lcal \tilde a)\\&=&e^{it\frac{\nu_j^2-\overline{\nu_k}^2}2} W_{j, k}(\tilde a)\\&=& W_{j, k}(V^t \tilde a)=PS_{\nu_j, -\nu_k}(\lcal\circ V^t \tilde a).
\end{eqnarray*} The identity holds for all $j, k$ (everything is well defined even for low eigenvalues of the laplacian).

From this, we can deduce the following~:
\begin{itemize}
\item we can define $\lcal_\Gamma a\in \hcal_{PS}$ by $PS^\Gamma_{\nu_j, -\nu_k}(\lcal_\Gamma a)=PS_{\nu_j, -\nu_k}(\lcal \tilde a)$, and this definition does depend on the choice of $\tilde a$.
\item the adjoint $\lcal_{\Gamma\sharp}$ sends $PS^\Gamma_{\nu_j, -\nu_k}$ to $W^\Gamma_{j, k}$, hence
$\lcal_{\Gamma\sharp}$ is an isometry from $\hcal_{PS}^*$ to $\hcal_{W}^*$, and
$\lcal_\Gamma$ is an isometry from $\hcal_{W}$ to $\hcal_{PS}$.
\item the family $PS^\Gamma_{\nu_j, -\nu_k}$ forms an independent family in $\hcal_{PS}^*$~:
if $\sum_{j, k}\alpha_{jk}PS^\Gamma_{\nu_j, -\nu_k}=0$ with $\sum|\alpha_{jk}|^2<+\infty$, then
$\alpha_{jk}=0$. This comes from the fact that the $W_{j, k}$ form an independent family in $\hcal_W^*$, and $\lcal_{\Gamma\sharp}PS^\Gamma_{\nu_j, -\nu_k}=W^\Gamma_{j, k} $. It follows that the family $PS^\Gamma_{\nu_j, -\nu_k}$ is an orthonormal basis of
$\hcal_{PS}^*$.
\item for $a\in \Pi\scal_{\omega}^\infty$, we have
$$\lcal_\Gamma\circ V^t_\Gamma a=G_\Gamma^t\circ \lcal_\Gamma a,$$
as an equality between two elements of $\hcal_{PS}$.
\end{itemize}

\end{proof}

We note that Hilbert-Schmidt symbols are a very special class of
symbols. But it is mainly a condition on the $r$-growth of
symbols. By multiplying by powers of $r$, the intertwining
operator and the time evolution extend readily to more general
symbols.

\section{Further discussion about the Hilbert spaces $\hcal_W$ and $\hcal_{PS}$\label{HILBERT}}

We now discuss the elements of the Hilbert spaces $\hcal_W$ and
$\hcal_{PS}$ in more detail~: we describe sufficient conditions
for a function to belong to $\hcal_W$ and $\hcal_{PS}$, in terms
of regularity and decay. We use the regularity properties of the
boundary values $T_{\nu_j}$ of eigenfunctions, described by Otal
\cite{O}. In the automorphic case, the regularity properties can
be read off directly from the automorphy equation (see e.g.
\cite{MS,MS2}).

\subsection{H\"older continuity of $T_{\nu_j}$\label{s:Otal}}
Following Otal \cite{O}, we say that a function $F$ defined on $\R$ is $2\pi$-periodic if there is a constant $C$ such that $F(x+2\pi)=F(x)+C$, for all $x$. If $F$ is locally integrable, its derivative $DF$ yields a well-defined distribution on $S^1=\R/2\pi\Z$,
$$DF(\varphi)=-\int_0^{2\pi}\frac{\partial \varphi}{\partial \theta} F(\theta)d\theta+\varphi(0)\left[F(2\pi)-F(0)\right],$$
for every smooth function $\varphi$ on $S^1$.
For $0
\leq \delta \leq 1$ we say that a $2 \pi$-periodic function $F: \R
\to \C$ is $\delta$-H\"older if $|F(\theta) - F(\theta')| \leq C
|\theta - \theta'|^{\delta}.$ The smallest constant is denoted
$||F||_{\delta}$. We denote the  Banach space of $\delta$-H\"older
functions with norm $||F||_{\delta}$ by $\Lambda_{\delta}$.

We recall:
\begin{thm} \label{O} (\cite{O}) Suppose that $\phi$ is a laplacian
eigenfunction of eigenvalue $-s (1 - s)=-\left(\frac14+r^2\right)$, with $s=
\frac{1}{2} + i r$ and $\Re e(s) \geq 0$. Assume that  $||
\phi||_{\infty} < \infty$ and $||\nabla
\phi||_{\infty} < \infty$. Then its Helgason boundary value $T_{s,
\phi}$ is the derivative of a $\Re e(s) $-H\"older periodic function $F$.

In addition, letting $\delta=\Re e (s)$, we have
$$\Vert F\Vert_\delta\leq \frac{C}{|C_s|}|s|\left(||
\phi||_{\infty} +||\nabla
\phi||_{\infty} \right),$$
where $C>0$ is an absolute constant, and
$C_s=\int_{0}^{+\infty}\int_0^{2\pi}e^{-(1+s)t} P^s(\tanh \frac{t}2, \theta)d\theta dt$; where $P$ is the Poisson kernel of the unit disc.
\end{thm}
Outside of the finite number of ``small eigenvalues'' of $\X$, we have $\Re e(
s) = \frac{1}{2}$ and hence $T_{\nu_j} $ is the derivative of a
H\"older $1/2$-continuous function.
The upper bound on $\Vert F\Vert_\delta$ given by Otal's proof is quite crude, but will be sufficient for our purposes.

The behavior
of $C_s$ for $s=\frac12+ir$ and $r\To \pm \infty$ can be evaluated
by the stationary phase method.  The calculation is routine but
not completely straightforward  because the domain of integration
is non- compact.  However, for the sake of brevity we omit the
details.
One finds that $C_s\sim Cr^{-1/2}$, with $C\not=0$.

 \subsection{\label{HSVSW} $\hcal_W$ and $L^2_W(G\times \R, dg\times dp(r))$}

 In this section, we clarify the relation between Hilbert-Schmidt
 inner product, which induces the inner product $\langle,
 \rangle_W$ on symbols, and $L^2_W(G\times \R, dg\times dp(r))$. The second term in the
 following
 proposition is the discrepancy between the $||\cdot||_W$ and the $L^2$ norm on symbols (again, we stress the fact that this discrepancy would also appear in a euclidean situation). We denote
 $\dcal $ a fundamental domain for the action of $\Gamma$ on $\D$.

\begin{prop} \label{HSX} Let $Op^{\Gamma}(a) $ be  a Hilbert-Schmidt pseudo-differential operator on $\X$ with
complete symbol $a$. Then $ ||a||_{W}^2\defi
||Op^{\Gamma}(a)||^2_{HS(\X)}$ is given by
$$\begin{array}{lll} ||a||_{W}^2 & = &  \int_{\dcal} \int_B \int_{\R_+}
|a(z, b, r)|^2 e^{\langle z, b \rangle} \Vol(dz) dp(r)db
 \\ && \\
 && +  \sum_{\gamma\in \Gamma\setminus \{e\} } \int_{z\in\dcal, (b, r)\in B\times \R_+}a(z, b,r) \overline{a(\gamma \cdot  z, b,r)}
e^{(\frac12+ir)\langle z, b \rangle}
e^{(\frac12-ir)\langle \gamma \cdot z, b \rangle} dp(r)db
\Vol(dz). \end{array}$$
\end{prop}

\begin{proof} Recall that $K_a^\Gamma(z, w)=\sum_\gamma K_a(z, \gamma w)$, and that $K_a$ is invariant by the diagonal action of $\Gamma$~: $K_a(\gamma z, \gamma w)=K_a(z,  w)$.

The composition formula for the kernels is~:
\begin{equation*}K_a^{\Gamma} \circ K_b^{\Gamma}(z, w)  = \int_{\D} K_a(z,
v) K_b^{\Gamma}(v,w) \Vol(dv).
\end{equation*}
Hence
$K_a^{\Gamma} \circ K_b^{\Gamma \dagger}(z, w) =  \int_{\D} K_a(z,
v) \overline{K_b^{\Gamma}(w,v)} \Vol(dv).$
  Taking the trace.
  \begin{equation*}
  Tr (K_a^{\Gamma} \circ K_b^{\Gamma\dagger})
 =\sum_{\gamma\in\Gamma }\int_{z\in\dcal}  \int_{v\in \D} K_a(z,
v) \overline{K_b(\gamma \cdot z,v)} \Vol(dv) \Vol(dz) .
\end{equation*}
The rest of the calculation proceeds as in Proposition \ref{HSD}, using the Fourier inversion formula.
The first term corresponds to $\gamma = e$ and the second term to
$\gamma \not= e$.

\end{proof}

\subsection{$\hcal_W$}

We now describe a large class of elements of  the Hilbert space
$\hcal_W$. That is, we determine sufficient conditions on $a $
so that $\Op_\Gamma(a)$ is Hilbert-Schmidt, or in terms of Wigner
distributions, so that
\begin{equation} ||a||_W^2 = \sum_{j, k} | W^\Gamma_{j,
k} (a)|^2< \infty.
\end{equation}

In the following, $\langle x \rangle = (1 + |x|^2)^{1/2}$. If $C$ is an operator, we define  $ad(\Lap) C = [\Lap, C]$. We denote $\lambda_j=-\left(\frac14+r_j^2\right)$ the laplacian eigenvalues.
The following is of course not optimal, but gives an adequate
idea of a large class of elements in $\hcal_W$.

\begin{prop} If $\sup_{r \in \R_+} \langle r \rangle^6
||a_r||_{C^6} < \infty$, then $a \in \hcal_W$. \end{prop}

\begin{proof}

It suffices to prove the following

\begin{lem} Let $$|||a||| \defi\sup_j \sup_{(z, b)\in\dcal} \langle \lambda_j \rangle^{2}\left[ |(I - Y^2)\Lap^2_z a |
+r_j |(I - Y^2)\Lap_z\nabla_z a |\\
+ r_j^2 |(I - Y^2)\nabla_z^2 a |+ r_j^2 |(I - Y^2)\nabla_z a |\right].$$

If $|||a||| < \infty$, then $a \in \hcal_W$.

\end{lem}

To prove this, we first note that, by  Weyl's law, $\sum_{j, k}
\langle \lambda_j \rangle^{-2} \langle \lambda_j - \lambda_k
\rangle^{-2} < \infty$ in dimension two.

We will also use the expansion of $\epsilon_{\nu_j}$ into $K$-Fourier series, which takes the form
$$\epsilon_{\nu_j}=\sum_{m\in\Z} \phi_{j, m},$$
with $Y\phi_{j, m}=2im \phi_{j, m}$. We use the fact that $\norm{\phi_{j, m}}_{L^2(\Gamma\backslash G)}=1$, proved in \cite{Z}
(the full definition of $\phi_{j, m}$ can be found in Proposition 2.2 of \cite{Z}, in particular, $\phi_{j, 0}=\phi_j$).

Let us write $B = \Op(b) \defi \left(ad(\Lap)^2 \Op(a) \right)\circ\Lap^2$.
Then,
$$\begin{array}{lll} W^\Gamma_{j, k}(b) & = & \sum_{m \in \Z} \langle b_{r_j},
\phi_{j} \phi_{k, m} \rangle\\ && \\
& = &  \sum_{m \in \Z} \langle 2m \rangle^{-2}  \langle b_{r_j}, (I
-
Y^2) \phi_{j} \phi_{k, m} \rangle \\ &&\\
& = & \sum_{m \in \Z} \langle 2m \rangle^{-2}  \langle (I - Y^2)
b_{r_j},
\phi_{j} \phi_{k, m} \rangle \\ && \\
& \leq & \sup |(I - Y^2) b_{r_j}| \sum_{m \in \Z} \langle 2m
\rangle^{-2}
\langle  |\phi_{j}|,  |\phi_{k, m}| \rangle \\ && \\
& \leq & C  \sup |(I - Y^2) b_{r_j}|, \end{array} $$ where $C$ is
a uniform constant. Here we use that $ \langle  |\phi_{j}|,
|\phi_{k, m}| \rangle \leq 1$ by the Schwartz inequality and
the fact that $||\phi_{j}||_{L^2} = ||\phi_{k, m} ||_{L^2} =
1$.

It follows that
\begin{multline*}  |W^\Gamma_{j, k} (b) |
\leq C \; \langle \lambda_j \rangle^{-2} \langle \lambda_j -
\lambda_k
\rangle^{-2} \sup |(I - Y^2) b_{r_j}| \\
 \leq C \; \langle \lambda_j \rangle^{-2} \langle \lambda_j -
\lambda_k \rangle^{-2} \\
 \sup_{(z, b)\in\dcal} \langle \lambda_j \rangle^{2}[ |(I - Y^2)\Lap^2_z a |
+r_j |(I - Y^2)\Lap_z\nabla_z a |\\
+ r_j^2 |(I - Y^2)\nabla_z^2 a |+ r_j^2 |(I - Y^2)\nabla_z a |]
\end{multline*}

Here, we use that the complete symbol of $ \Op(a)\circ\Lap^2 $ is $(\frac14+r^2)^2a(z, b, r)$. Further the complete symbol of $ad(\Lap) \Op(a)$ is given by $\Lap_z a+\left(\frac12 +ir\right)\nabla_z a. \nabla_z\la z, b\ra$.
\end{proof}

\subsection{The Hilbert space $\hcal_{PS}$}

We now consider the analogous question of conditions on $a$
so that
\begin{equation} \label{PSIPa} ||a||^2_{PS} \defi \sum_{j, k} | {PS}^\Gamma_{\nu_j,
-\nu_k} (a) |^2 < \infty.
\end{equation}

\begin{prop} If $a$ is $\Gamma$ automorphic and  $\sup_{r} \la r\ra^{12}  ||a_r ||_{C^{3}} < +\infty$, then $a \in \hcal_{PS}$.
\end{prop}

This follows from

\begin{lem} We have,
 for any $M$, \begin{multline*} ||a||_{PS}^2 \leq
\sum_{j, k} |\nu_j|^{3/2}  |\nu_k|^{3/2}(\norm{\phi_j}_\infty + \norm{\nabla\phi_j}_\infty)
(\norm{\phi_k}_\infty + \norm{\nabla\phi_k}_\infty)
 \langle \nu_j+\overline{\nu_k}\rangle^{-M}
 \\
\sup_{(b', b)\in
B\times B} \bigg[|b-b^\prime|^{-(1+\nu_j-\overline{\nu_k})} )
 \rcal \la \partial t\ra^M a_{\frac{\nu_j-\overline{\nu_k}}{2i}}(b', b,
i(\nu_j+\overline{\nu_k})) , \\
 \frac{\partial}{\partial b'} |b-b^\prime|^{-(1+\nu_j-\overline{\nu_k})} )
 \rcal \la \partial t\ra^M a_{\frac{\nu_j-\overline{\nu_k}}{2i}}(b', b,
i(\nu_j+\overline{\nu_k})) ,\\ \frac{\partial}{\partial b}|b-b^\prime|^{-(1+\nu_j-\overline{\nu_k})} )
 \rcal \la \partial t\ra^M a_{\frac{\nu_j-\overline{\nu_k}}{2i}}(b', b,
i(\nu_j+\overline{\nu_k})),  \\
\frac{\partial^2}{\partial b\partial b'}|b-b^\prime|^{-(1+\nu_j-\overline{\nu_k})} )
 \rcal \la \partial t\ra^M a_{\frac{\nu_j-\overline{\nu_k}}{2i}}(b', b,
i(\nu_j+\overline{\nu_k}))\bigg]
\end{multline*}
\end{lem}

\begin{proof}
We use the relation
\begin{equation}  PS^\Gamma_{\nu_j, -\nu_k} (a) =\int  \frac{1}{ |b - b'|^{1 + \nu_j-\overline{\nu_k}} } \rcal \chi a_{\frac{\nu_j-\overline{\nu_k}}{2i}}(b', b,
i(\nu_j+\overline{\nu_k})) T_{\nu_j}(db)  \overline{T_{-\nu_k}}(db'),
 \end{equation}
 which is obtained from \eqref{PSRADON}. Since $\chi a_{\frac{\nu_j-\overline{\nu_k}}{2i}}$ is compactly supported on $G$, then the Radon-Fourier transform $ \rcal a$ is compactly supported in the variables $(b', b)\in B^{(2)}$, so the singularity of $|b-b'|$ on the diagonal is not a problem. It follows by repeated integration by parts in
$\partial t$ that if $a \in C_c^M(G)$, then $\rcal a(b, b', i(\nu_j+\overline{\nu_k})) =
O(\langle \nu_j+\overline{\nu_k}\rangle^{-M})$.

Let us call $F_{\nu_j}$ the H\"older function such that $T_{\nu_j}=F'_{\nu_j}$, in the sense of \S \ref{s:Otal}. We use the formula
\begin{eqnarray*}\int \varphi(b', b) T_{\nu_j}(db)\overline{T_{-\nu_k}(db')}&=&\varphi (0, 0)[F_{\nu_j}(2\pi)-F_{\nu_j}(0)]\overline{[F_{-\nu_k}(2\pi)-F_{-\nu_k}(0)]}\\&&- [F_{\nu_j}(2\pi)-F_{\nu_j}(0)]\int \frac{\partial}{\partial b'}\varphi(b', 0) \overline{F_{-\nu_k}}(b')db'\\&&-  \overline{[F_{-\nu_k}(2\pi)-F_{-\nu_k}(0)]}\int \frac{\partial}{\partial b}\varphi(0, b) {F_{\nu_j}}(b)db\\
&&+\int \frac{\partial^2}{\partial b\partial b'}\varphi(b', b){F_{\nu_j}}(b)\overline{F_{-\nu_k}}(b') db\, db',
\end{eqnarray*}
valid for every smooth function $\varphi$ on $B\times B$.

It follows that
\begin{multline*}|\int \varphi(b', b) T_{\nu_j}(db)\overline{T_{-\nu_k}(db')}|\\
\leq \norm{F_{\nu_j}}_{\delta_j} \norm{F_{-\nu_k}}_{\delta_k}\sup_{(b', b)\in
B\times B} \left(|\varphi(b', b)|,  |\frac{\partial}{\partial b'}\varphi(b', b)| ,|\frac{\partial}{\partial b}\varphi(b', b)|, | \frac{\partial^2}{\partial b\partial b'}\varphi(b', b)|\right)
\end{multline*}
where $\delta_j=\frac12+\Re e(\nu_j)$, and the H\"older norm $\norm{.}_\delta$ is the one appearing in Theorem \ref{O}.

We can then write
\begin{multline*}| PS_{\nu_j, -\nu_k} (a)|\leq  \langle \nu_j+\overline{\nu_k}\rangle^{-M}\norm{F_{\nu_j}}_{\delta_j} \norm{F_{-\nu_k}}_{\delta_k}\\
\sup_{(b', b)\in
B\times B} \bigg[|b-b^\prime|^{-(1+\nu_j-\overline{\nu_k})} )
 \rcal \la \partial t\ra^M \chi a_{\frac{\nu_j-\overline{\nu_k}}{2i}}(b', b,
i(\nu_j+\overline{\nu_k})) , \\
 \frac{\partial}{\partial b'} |b-b^\prime|^{-(1+\nu_j-\overline{\nu_k})} )
 \rcal \la \partial t\ra^M \chi a_{\frac{\nu_j-\overline{\nu_k}}{2i}}(b', b,
i(\nu_j+\overline{\nu_k})) ,\\ \frac{\partial}{\partial b}|b-b^\prime|^{-(1+\nu_j-\overline{\nu_k})} )
 \rcal \la \partial t\ra^M \chi a_{\frac{\nu_j-\overline{\nu_k}}{2i}}(b', b,
i(\nu_j+\overline{\nu_k})),  \\
\frac{\partial^2}{\partial b\partial b'}|b-b^\prime|^{-(1+\nu_j-\overline{\nu_k})} )
 \rcal \la \partial t\ra^M \chi a_{\frac{\nu_j-\overline{\nu_k}}{2i}}(b', b,
i(\nu_j+\overline{\nu_k}))\bigg]
\end{multline*}

 By  Theorem \ref{O},
 $$\norm{F_{\nu_j}}_{\delta_j}= O(|\nu_j|^{3/2})(\norm{\phi_j}_\infty + \norm{\nabla\phi_j}_\infty).$$
Moreover, by the well-known local Weyl law estimates,
$\norm{\phi_j}_\infty=O(|\nu_j|^{\half})$  and $
\norm{\nabla\phi_j}_\infty=O(|\nu_j|^{\frac{3}{2}}).$
We find
\begin{equation*}| PS_{\nu_j, -\nu_k} (a)|\leq  \langle \nu_j+\overline{\nu_k}\rangle^{-M} |\nu_j|^3|\nu_k|^3\\
\max\left(|\nu_j|, |\nu_k|\right) \norm{ a_{\frac{\nu_j-\overline{\nu_k}}{2i}}}_{C^{M+1}}.
 \end{equation*}

Using the Weyl law in dimension $2$, $|\lambda_j|\sim C j$, one sees that the series
$$\sum_{j, k}\langle r_j -r_k\rangle^{-M} |r_j|^3 |r_k|^3
\max\left(|r_j|, |r_k|\right) \left \la\frac{r_j+r_k}{2}\right\ra^{-N}$$
converges for $M>1$ and $N>11$. The result follows.

 \end{proof}

  We stress again the fact that there is nothing optimal in this upper bound.

  \section{\label{APPENDIX} Appendix}

  In this section we sketch the proof of Theorem \ref{mainintro2}. We closely follow the proof  in Section 4 of \cite{AZ}.

  By the generalized Poisson formula and the definition of
$\Op(a)$,
\begin{equation}\label{e:WWW}\langle \Op_\Gamma(a) \phi_{ir_j}, \phi_{ir_k}\rangle = \int_{B \times B} \left( \int_{\D} \chi a(z,b) e^{(\frac12 +ir_j) \langle z, b
\rangle} e^{(\frac12+ir_k) \langle z, b' \rangle
}\Vol(dz)\right)T_{ir_j}(db) \overline{T_{-ir_k}(db')}.
\end{equation}

Here we are only interested in real values of $r_j, r_k$, since we consider the asymptotics $r_j\to +\infty$ and $|r_j-r_k|$ bounded.
  We apply stationary phase to the simplify the inner $\D$
  integral. More precisely, in \cite{AZ} and in this article, we rewrite the integral in the form
  $$\langle \Op_\Gamma(a) \phi_{ir_j}, \phi_{ir_k}\rangle = 2^{(1+ir_j+ir_k)}\int L_{ir_k} \chi a (b', b, \tau) PS_{ir_j, -ir_k }(db', db, d\tau),$$
as was shown in Theorem \ref{t:WPS}, and we then replace $L_{ir_k} \chi a (b', b, \tau)$ by its expansion into powers of $r_k^{-1}$, obtained by the method of stationary phase.

 There is one detail that we did not discuss in \cite{AZ}, and that was mentioned to us by Michael Schr\"oder (see \cite{SchDiss}). The $PS$-distributions have a singularity of the form $|b-b'|^{-(1+ir_j+ir_k)}$ on the diagonal ($b'=b$), and thus can only be integrated along functions that vanish on a neighbourhood of the diagonal. The function $L_{ir_k} \chi a$ does not satisfy this condition, and it is for a very special reason that its integral along $PS_{ir_j, -ir_k }$ can be defined~: its singularity exactly cancels with $|b-b'|^{-(1+ir_j+ir_k)}$. However, when replacing $L_{ir_k} \chi a$ by its stationary phase expansion, one would have to justify the fact that each term, including the remainder term, can be integrated along $PS_{ir_j, -ir_k }$. This is not easy and wasn't discussed in \cite{AZ}.

It is actually simpler to carry out the localization step away from the diagonal with the original
inner integral \eqref{e:WWW}.
We see that the function
\begin{equation}\label{e:ipp} \int_{\D} \chi a(z,b) e^{(\frac12 +ir_j)
\langle z, b \rangle} e^{(\frac12+ir_k) \langle z, b' \rangle
}\Vol(dz)\end{equation}
is integrated along $T_{ir_j}(db) \overline{T_{-ir_k}(db')}$, and with the latter there is no issue on the diagonal.

The critical set in the oscillatory integral \eqref{e:ipp} occurs where $\nabla \langle z, b
\rangle = - \nabla \langle z, b' \rangle$. So $z \in
\gamma_{b',b}$. There is a neighbourhood $V$ of the diagonal such that $\gamma_{b',b}$ does not intersect the support of $\chi a$ for $(b', b)\in V$.
We take a smooth function $f$ on $B\times B$, supported in $V$, that is identically $1$ on a neighbourhood of the diagonal, and divide the integral \eqref{e:WWW} into
\begin{multline*}\int_{B \times B} f(b', b)\left( \int_{\D} \chi a(z,b) e^{(\frac12 +ir_j) \langle z, b
\rangle} e^{(\frac12+ir_k) \langle z, b' \rangle
}\Vol(dz)\right)T_{ir_j}(db) \overline{T_{-ir_k}(db')} \\ +
\int_{B \times B} (1-f(b', b))\left( \int_{\D} \chi a(z,b) e^{(\frac12 +ir_j) \langle z, b
\rangle} e^{(\frac12+ir_k) \langle z, b' \rangle
}\Vol(dz)\right)T_{ir_j}(db) \overline{T_{-ir_k}(db')}.
\end{multline*}

For the first term, the phase has no critical point, and we integrate by parts using
$$\frac{1}{|\nabla_z \langle z, b \rangle - \nabla_z \langle z, b'
\rangle|^2} \left( \nabla_z \langle z, b \rangle - \nabla_z
\langle z, b' \rangle \right) \cdot \nabla. $$ Since $T_{ir_j}, T_{-ir_k}$ have polynomial bounds in $r_j,r_k$, repeated partial
integration shows that this first integral is
$O(\langle r_k \rangle^{-\infty}). $

The second term, because of the cut-off $(1-f(b', b))$, is now supported away from the diagonal,
and can be rewritten as
$ 2^{(1+ir_j+ir_k)}\int (1-f(b', b)) L_{ir_k} \chi a (b', b, \tau) PS_{ir_j, -ir_k }(db', db, d\tau).$
The proof of Section 4 in \cite{AZ} now applies without problem.

\end{document}